\documentclass[a4paper, reqno]{amsart}
\allowdisplaybreaks

\title[Connecting $\mathcal{W}$-algebras on $\sll_4$]{Connecting affine $\mathcal{W}$-algebras: A case study on $\sll_4$}

\author[Justine Fasquel]{Justine Fasquel}
\address[J.F.]{School of Mathematics and Statistics, University of Melbourne, Parkville, Australia, 3010}
\email{justine.fasquel@unimelb.edu.au}

\author[Zachary Fehily]{Zachary Fehily}
\address[Z.F.]{School of Mathematics and Statistics, University of Melbourne, Parkville, Australia, 3010}
\email{fehilyz@unimelb.edu.au}

\author[Ethan Fursman]{Ethan Fursman}
\address[E.F.]{School of Mathematics and Statistics, University of Melbourne, Parkville, Australia, 3010}
\email{ethan.fursman@unimelb.edu.au}

\author[Shigenori Nakatsuka]{Shigenori Nakatsuka}
\address[S.N.]{Research Institute for Mathematical Sciences, Kyoto University, Kyoto 606-8502, Japan}
\email{shigenori.nakatsuka.2022@gmail.com}

\usepackage{array}
\usepackage{latexsym}
\usepackage{amsmath}
\usepackage{amsfonts}
\usepackage{amssymb,nccmath}
\usepackage{mathtools}
\usepackage{amsxtra}
\usepackage{mathrsfs}
\usepackage{eucal}
\usepackage[all]{xy}
\usepackage{color}
\usepackage{braket}
\usepackage{graphics, setspace}
\usepackage{etoolbox, textcomp}
\usepackage{comment}
\usepackage{changepage}
\usepackage{xcolor}
\definecolor{rouge}{rgb}{0.85,0.1,.4}
\definecolor{bleu}{rgb}{0.1,0.2,0.9}
\definecolor{violet}{rgb}{0.7,0,0.8}
\usepackage[colorlinks=true,linkcolor=bleu,urlcolor=violet,citecolor=rouge]{hyperref}
\usepackage{cleveref}
\usepackage{lscape}
\usepackage{caption}
\usepackage{subcaption}
\usepackage{enumitem}
\usepackage{graphicx}
\usepackage{arydshln}
\usepackage{tikz} \usetikzlibrary{calc} \tikzset{>=latex} \usetikzlibrary{backgrounds} \usetikzlibrary{shapes.geometric}
\usepackage{tikz-cd}
\usepackage[colorinlistoftodos]{todonotes}
\usepackage{bbm}
\usepackage{blindtext}

\usepackage[backend=biber,safeinputenc,sorting=nyt,maxbibnames=10]{biblatex}
\usepackage[aligntableaux=center]{ytableau}
\ytableausetup{mathmode,boxsize=0.4em}

\newtheorem{definition}{Definition}[section]
\newtheorem{proposition}[definition]{Proposition}
\newtheorem{theorem}[definition]{Theorem}

\newtheorem{ThmLetter}{Theorem}
\newtheorem{corollary}[definition]{Corollary}
\newtheorem{CorLetter}{Corollary}
\newtheorem{lemma}[definition]{Lemma}

\theoremstyle{remark}
\newtheorem{remark}[definition]{Remark}

\numberwithin{equation}{section}

\newcommand{\Z}{\mathbb{Z}}
\newcommand{\Q}{\mathbb{Q}}
\newcommand{\C}{\mathbb{C}}

\newcommand{\D}{\mathcal{D}}
\newcommand{\W}{\mathcal{W}}
\newcommand{\V}{\mathcal{V}}

\newcommand{\kk}{\mathsf{k}}   

\newcommand{\g}{\mathfrak{g}}
\newcommand{\h}{\mathfrak{h}}
\newcommand{\n}{\mathfrak{n}}

\newcommand{\nil}{\mathfrak{n}}

\newcommand{\sll}{\mathfrak{sl}}
\newcommand{\gl}{\mathfrak{gl}}
\newcommand{\SL}{\mathrm{SL}}

\newcommand{\fH}[1]{\mathrm{H}_{\,\ydiagram{#1}\,}}                                         
\newcommand{\HH}{\mathrm{H}}
\newcommand{\OO}{\mathbb{O}}

\newcommand{\fHdeg}[2]{\mathrm{H}^{#2}_{\,\ydiagram{#1}\,}}
\newcommand{\OHdeg}[2]{\mathrm{H}^{#2}_{\,{#1}\,}} 
\newcommand{\fd}[1]{d_{\,\ydiagram{#1}\,}}                                                  

\newcommand{\pW}[1]{\W^{\kk}(\,\ydiagram{#1}\,)}
\newcommand{\pWlevel}[2]{\W^{#2}(\,\ydiagram{#1}\,)}

\newcommand{\pBRST}[1]{C^\bullet_{\ydiagram{#1}}}
\newcommand{\pWak}[2]{\mathbb{W}^{\kk}_{{#2}}(\,\ydiagram{#1}\,)}
\newcommand{\affWak}[1]{\mathbb{W}^{\kk}_{{#1}}}
\newcommand{\presol}[2]{\mathrm{C}^{\kk}_{{#2}}(\,\ydiagram{#1}\,)}
\newcommand{\heis}{\pi^{\kk+4}_\mathfrak{h}}
\newcommand{\Fock}[1]{\pi^{\kk+4}_{\mathfrak{h},{#1}}}
\newcommand{\bg}{\beta\gamma}
\newcommand{\ff}[1]{\beta\gamma^{\{#1\}}}
\newcommand{\s}[2]{S_{{#1}}^{\,\ydiagram{#2}\,}}
\newcommand{\ts}[2]{\tilde{S}_{{#1}}^{\,\ydiagram{#2}\,}}

\newcommand{\df}[1]{#1{}^\natural}
\newcommand{\stgen}[1]{\langle #1 \rangle}
\newcommand{\bt}[1]{\beta_{\mathrm{#1}}}
\newcommand{\gm}[1]{\gamma_{\mathrm{#1}}}
\newcommand{\longstick}[2]{\left. {#1} \right|_{#2}}
\newcommand{\cl}[1]{\overline{#1}}
\newcommand{\lbr}[2]{\{ #1 {}_\lambda #2 \}}
\newcommand{\finW}[1]{\mathscr{U}(\, \ydiagram{#1} \,)}

\renewcommand{\ker}{\operatorname{Ker}}

\newcommand{\hlat}{\Pi}                          		                                    

\newcommand{\nosimple}[1]{#1}  
\newcommand{\intNOz}[1]{\int Y(\nosimple{#1},z) \ \mathrm{d}z}                                         
\newcommand{\scrNO}[2]{\intNOz{{#2}\mathrm{e}^{-\frac{1}{\kk+4}\alpha_{#1}}}}                 

\newcommand{\fockIO}[1]{\mathrm{e}^{-\frac{1}{\kk+4}\alpha_{#1}}}
\newcommand{\ee}{\mathsf{e}}                                                                
\newcommand{\sFMS}{S_\text{FMS}}                                                            
\newcommand{\wun}{\mathbbm{1}}                                                              
\newcommand{\dd}{\mathrm{d}}

\newcommand{\Span}{\operatorname{Span}}
\newcommand{\wt}{\operatorname{wt}}
\newcommand{\gr}{\mathrm{gr}}

\newcommand{\KL}{\mathbf{KL}}
\newcommand{\weyl}{\mathbb{V}}

\newcommand{\hwt}[1]{\mathrm{e}^{{#1}}}

\newcommand{\EM}[1]{E_{#1}}             

\newcommand{\coffaa}{{\beta_1}}
\newcommand{\coffab}{\beta_2-\gamma_1\beta_4}
\newcommand{\coffac}{\gamma_4}

\newcommand{\coffba}{\wun}
\newcommand{\coffbb}{\beta_2}
\newcommand{\coffbc}{\beta_3-\gamma_2\beta_5}

\newcommand{\coffca}{\beta_1}
\newcommand{\coffcb}{-\gamma_1+\gamma_3}
\newcommand{\coffcc}{\beta_3}

\newcommand{\coffda}{\beta_1}
\newcommand{\coffdb}{\wun-\gamma_1\gamma_3}
\newcommand{\coffdc}{\beta_3}

\newcommand{\coffea}{\wun}
\newcommand{\coffeb}{\wun}
\newcommand{\coffec}{\beta_3}

\newcommand{\cofffa}{\wun}
\newcommand{\cofffb}{\beta_2}
\newcommand{\cofffc}{\gamma_2}

\newcommand{\coffga}{\wun}
\newcommand{\coffgb}{\wun}
\newcommand{\coffgc}{\wun}

\newcommand{\bideg}{\text{bideg}}
\DeclareRobustCommand{\rchi}{{\mathpalette\irchi\relax}}
\newcommand{\irchi}[2]{\raisebox{\depth}{$#1\chi$}}
\newcommand{\std}[1]{d^{\text{st}}_{\,\ydiagram{#1}\,}}                                     
\newcommand{\chid}[1]{d^{\rchi}_{\,\ydiagram{#1}\,}}                                             

\DeclareFontFamily{U}{mathx}{}
\DeclareFontShape{U}{mathx}{m}{n}{<-> mathx10}{}
\DeclareSymbolFont{mathx}{U}{mathx}{m}{n}
\DeclareMathAccent{\widecheck}{0}{mathx}{"71}

\newcommand{\numtableaux}[1]{
\ytableausetup{boxsize=1.15em,aligntableaux=bottom}
\begin{ytableau}
    #1
\end{ytableau}
\ytableausetup{boxsize=0.4em}}

\newcommand{\coloredtableaux}[2]{
{\ytableausetup{centertableaux,boxsize=1em,aligntableaux=bottom}
\ytableaushort{}
* {#1}
* [*(lightgray)]{#2} \ytableausetup{boxsize=0.4em}}}

\newcommand\doi[2]{\href{http://dx.doi.org/#1}{#2}}

\begin{document}

\begin{abstract}
    We introduce the partial reductions and inverse Hamiltonian reductions between affine $\W$-algebras along the closure relations of associated nilpotent orbits in the case of $\sll_4$, fulfilling all the missing constructions in the literature.
    We also apply the partial reductions to modules in the Kazhdan--Lusztig category and show compatibility with the usual reductions of Weyl modules.
\end{abstract}

\maketitle

\section{Introduction}

Affine $\W$-algebras are families of vertex algebras associated with nilpotent elements (orbits) of simple finite-dimensional Lie algebras $\g$.
In the early 90s, Feigin--Frenkel \cite{FF90} introduced the principal $\W$-algebras, associated to regular nilpotent elements, as quantum Hamiltonian reductions -- called quantized Drinfeld--Sokolov reductions -- of the affine vertex algebras $\V^\kk(\g)$.
Later this construction was extended to arbitrary nilpotent elements by Kac--Roan--Wakimoto \cite{KRW03} resulting in the definition commonly used today:
if $\OO$ a nilpotent orbit of $\g$ and $\kk\in\C$ a complex parameter, the $\W$-algebra associated to $\OO$ at level $\kk$ is given by
\begin{equation}
    \W^\kk(\OO):=\OHdeg{\OO}{}(\V^\kk(\g)).
\end{equation}

The quantum Hamiltonian reductions are naturally upgraded to functors from the category of modules over the affine vertex algebras to the category of the modules over the corresponding affine $\W$-algebras but they are difficult to control in general.
A promising approach has emerged recently to remedy this.
The rough idea lies in splitting the reduction functors into some more fundamental ones and applying these successive partial reductions which are better understood. This process is known as \emph{reduction by stages}.
The terminology is inspired by Morgan's PhD thesis \cite{Morgan14} 
and Genra--Juillard \cite{GJ23} who worked on equivalent technology in the context of finite $\W$-algebras.

Regarding the affine analogue, the first explicit example of reduction by stages for the affine $\W$-algebras goes back to an earlier work in 90s of Madsen--Ragoucy \cite{MR97} who showed that the principal $\W$-algebra of $\sll_3$ can be obtained as a certain reduction of the Bershadsky--Polyakov vertex algebra, which is the $\W$-algebra of $\sll_3$ associated with the minimal nilpotent orbit.

More generally, there is a partial order on nilpotent orbits of a Lie algebra $\g$ given by the inclusion of their closures for the Zarisky topology, that is $\OO_1 \leq \OO_2$ if $\OO_1\subset\overline{\OO}_2$. This lifts to a natural partial order on the family of $\W$-algebras associated to the same Lie algebra $\g$.
In type $A$, i.e. $\g\simeq\sll_n$, the nilpotent orbits are parametrized by partitions of $n$ and the partial order matches the dominance order on these partitions.
It is believed that the reduction by stages respects this partial order in the following sense: if $\OO_1 \leq \OO_2$, there should exist a functor reducing the $\W$-algebra associated with $\OO_1$ to the one with $\OO_2$.
Such examples supporting this conjecture were obtained recently in \cite{CFLN}.

The phenomenon of reduction by stages seems to be remarkably ubiquitous in closely related areas.
For instance, similar decompositions arise in the description of Whittaker models for the general linear groups over local fields in number theory \cite{GGS17}.
It is also the key ingredient in the construction of \emph{webs of $\W$-algebras} by Proch\'azka--Rap\v{c}\'ak \cite{PR18} in physics, which describe vertex algebras as chiral algebras appearing in the two-dimensional junctions of supersymmetric interfaces in the four-dimensional $N=4$ super Yang-Mills gauge theories.
A broader perspective might be better understood from the connection with the quantum geometric Langlands program \cite{FG20}.

The quantum Hamiltonian reductions from $\V^\kk(\g)$ to $\W$-algebras are known to behave nicely \cite{Ara15a}. This comes from the freeness of the Hamiltonian action appearing in the classical limit.
Therefore, it might be reasonable to hope for a way to recover $\V^\kk(\g)$ itself from the $\W$-algebras. 
It is commonly called \emph{inverse Hamiltonian reductions} and formulated as (conformal) embeddings of $\W$-algebras into tensor products of other $\W$-algebras and free field algebras when the associated nilpotent orbits satisfy the closure relation.

Inverse Hamiltonian reductions have a long history which takes its origins in physics.
The first embedding was described by Semikhatov \cite{Sem94} in string theory and relates the affine vertex algebra $\V^\kk(\sll_2)$ and the Virasoro algebra, i.e. the principal $\W$-algebra of $\sll_2$ \cite{Ad19}. 
More examples were worked out recently in small ranks \cite{ACG23, AKR21, CFLN} and for hook-type $\W$-algebras \cite{FN, Fehily2306.14673, Fehily23}; a general statement for type $A$ has been announced in \cite{Butson23}.

One key reason for this sudden surge of interest is the study of non-semisimple modular tensor categories appearing in the representation theory of $\W$-algebras. 
Indeed, it has been exemplified in the last decade that the inverse Hamiltonian reduction provides large representation categories of modules over $\W$-algebras from modular tensor categories coming from the so-called exceptional $\W$-algebras and non-semisimple modular tensor categories from free field algebras. We refer, among many others, to \cite{Ad19, ACG23, AKR, FRR, FR22} for more on this topic. 

\vspace{1ex}

\subsection*{Main results}
The present paper aims to study the partial and inverse Hamiltonian reductions connecting all $\W$-algebras associated with $\sll_4$ with respect to the poset of the nilpotent orbits (see Figures~\ref{fig:partial HR_sl4} and \ref{fig:partial inverse HR_sl4} below). 
As we are working in type $A$, the nilpotent orbits are parametrized by partitions and we label them using Young tableaux for convenience. For instance, the $\W$-algebra of $\sll_4$ associated with the subregular nilpotent orbit, corresponding to the partition $(3,1)$ is written as
\begin{equation}
    \pW{1,3}:=\fHdeg{1,3}{0}(\V^\kk(\sll_4))
\end{equation}
We first prove that applying an appropriate reduction to a $\W$-algebra results in another $\W$-algebra corresponding to a larger nilpotent orbit. More precisely, we have the following:
\begin{ThmLetter}\label{thm:main_isom}
For all levels $\kk\in \C$, there exist isomorphisms of vertex algebras
\begin{align*}
    &\fHdeg{1,2}{n}(\pW{1,1,2})\simeq \delta_{n,0}\pW{1,3},\qquad \fHdeg{2}{n}(\pW{1,3})\simeq \delta_{n,0}\pW{4},\\
    &\fHdeg{2}{n}(\pW{1,1,2})\simeq \delta_{n,0}\pW{2,2},\qquad \fHdeg{2}{n}(\pW{2,2})\simeq \delta_{n,0}\pW{1,3}.\\
\end{align*}
\end{ThmLetter}

The isomorphisms in Theorem~\ref{thm:main_isom} can be summarized in Figure~\ref{fig:partial HR_sl4}, where plain arrows describe reductions between two adjacent nilpotent orbits. In our case, they are all of \emph{Virasoro type}, denoted $\fHdeg{2}{\bullet}$, in the sense that the differential defining the partial reduction has the same form as the one defining the reduction of $\V^\kk(\sll_2)$ into the Virasoro vertex algebra.
In addition, we give an example of a partial reduction between two non-adjacent nilpotent orbits, $\ydiagram{1,1,2}$ and $\ydiagram{1,3}$, in dashed in the diagram. For a similar reason as before, this reduction is said to be of \emph{Bershadsky--Polyakov type}, denoted $\fHdeg{1,2}{\bullet}$. 

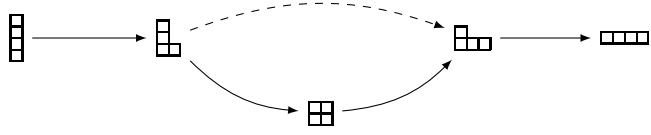
\begin{figure}[h]
    \centering
    \begin{tikzpicture}
			\node at (0, 2) (1111)  {\ydiagram{1,1,1,1}};
			\node at (2, 2) (112) {\ydiagram{1,1,2}};
			\node at (4, 1) (22) {\ydiagram{2,2}};
			\node at (6,2) (13) {\ydiagram{1,3}};
			\node at (8,2) (4) {\ydiagram{4}};
			\draw[->] (1111) to  (112);
			\draw[dashed, ->] (112) to [bend left=20] (13);
                          \draw[->] (13) to (4);
                          \draw[->] (112) to [bend right=20] (22);
                          \draw[->] (22) to [bend right=20] (13);
		\end{tikzpicture}
      \caption{Partial Hamiltonian reductions for $\sll_4$ $\W$-algebras.}
    \label{fig:partial HR_sl4}
\end{figure}

Theorem \ref{thm:main_isom} is proven using the explicit realisations of the $\W$-algebras associated with $\sll_4$ in terms of strong generators and OPEs and the direct description of the cohomology complexes.
Such computations would be difficult to generalize in higher ranks. However, for generic levels, we exhibit another proof of the isomorphisms for the 0-th cohomology based on the Wakimoto type free field realizations of $\W$-algebras, which boils down to comparing screening operators. 
We intend to explore this approach more generally in the near future.

As a consequence of Theorem \ref{thm:main_isom}, we obtain a parallel statement on the reduction by stages for finite $\W$-algebras $\mathscr{U}(\OO)$, which are obtained from $\W^\kk(\OO)$ by applying Zhu's functor \cite{Ara07, DK06}. 
\begin{CorLetter}\label{finite W-algebras}
There exist isomorphisms of associative algebras
\begin{align}
    \fHdeg{2}{0}(\finW{1,1,2})\simeq \finW{2,2},\quad \fHdeg{2}{0}(\finW{2,2})\simeq \finW{1,3}.
\end{align}
\end{CorLetter}
The second isomorphism and the remaining cases 
\begin{align}
    \fHdeg{1,2}{0}(\finW{1,1,2})\simeq \finW{1,3},\quad \fHdeg{2}{0}(\finW{1,3})\simeq \finW{4}
\end{align}
are already obtained in \cite[\S 4.1]{GJ23}. The first isomorphism seems to be new as far as we know, although it appears as a statement on Slodowy slices in the classical limit \cite{Morgan14}. 
We note that the first isomorphism completes the missing piece of the reduction by stages between finite $\W$-algebras for $\sll_4$.

As a natural counterpart to partial reduction, one may look at 
inverse Hamiltonian reductions, which are embeddings relating all the $\W$-algebras considered in this paper (see Figure~\ref{fig:partial inverse HR_sl4}).
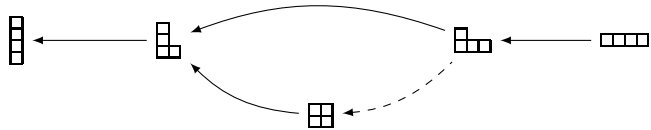
\begin{figure}[h]
    \centering
    \begin{tikzpicture}
			\node at (0, 2) (1111)  {\ydiagram{1,1,1,1}};
			\node at (2, 2) (112) {\ydiagram{1,1,2}};
			\node at (4, 1) (22) {\ydiagram{2,2}};
			\node at (6,2) (13) {\ydiagram{1,3}};
			\node at (8,2) (4) {\ydiagram{4}};
			\draw[<-] (1111) to  (112);
			\draw[<-] (112) to [bend left=20] (13);
                          \draw[<-] (13) to (4);
                          \draw[<-] (112) to [bend right=20] (22);
                          \draw[<-,dashed] (22) to [bend right=20] (13);
		\end{tikzpicture}
      \caption{Inverse Hamiltonian reductions.}
    \label{fig:partial inverse HR_sl4}
\end{figure}
In the case of $\sll_4$, almost all embeddings have been constructed by some of the authors \cite{CFLN, Fehily2306.14673, Fehily23}. 
We prove the existence of the remaining inverse reduction (denoted by a dashed arrow) in Theorem~\ref{iHR for non hook-type}. 
Using inverse Hamiltonian reduction between $\W$-algebras with adjacent nilpotent orbits, we obtain the following complete picture: 
\begin{ThmLetter}\label{mainthm:ihr}
Let $\lambda, \mu$ be two partitions of $4$ such that $\mathbb{O}_\lambda \geq \mathbb{O}_\mu$. Then there exists an embedding
\begin{equation} 
    \W^\kk (\OO_\mu) \hookrightarrow \W^\kk (\OO_\lambda) \otimes \beta \gamma^{N-M} \otimes \hlat^M
\end{equation}
where $M$ and $N$ are integers determined by $\lambda$ and $\mu$.
More precisely, 
\begin{equation}
    N=\frac{1}{2}(\dim \OO_\lambda-\dim \OO_\mu)
\end{equation}
and $M$ is given by the number of inverse Hamiltonian reductions used.
\end{ThmLetter}
The proof follows the idea of \cite{Fehily2306.14673} using automorphisms of free field algebras but with convenient sets of screening operators. For the non-hook type partition $\ydiagram{2,2}$ in particular, the partial reduction provides an unconventional set of screening operators which are very helpful for the construction of the inverse Hamiltonian reductions. 

In another direction, we show that the isomorphisms of vertex algebras obtained in Theorem~\ref{thm:main_isom} generalize to representations obtained from the Weyl modules $\weyl^\kk_\lambda$ $(\lambda\in P_+)$:
\begin{ThmLetter}\label{mainthm:modules}
For generic levels $\kk$, the $\W^\kk(\OO_f)$-modules $\HH_{\OO_f}^0(\weyl^\kk_\lambda)$ $(\lambda\in P_+)$ are simple and satisfies the following isomorphisms
\begin{align*}
\begin{array}{lll}
\fHdeg{2,2}{0}(\weyl_\lambda^k)\simeq \fHdeg{2}{0}(\fHdeg{1,1,2}{0}(\weyl_\lambda^k)),&& \fHdeg{1,3}{0}(\weyl_\lambda^k)\simeq \fHdeg{2}{0}(\fHdeg{2,2}{0}(\weyl_\lambda^k)),\\
&&\\
\fHdeg{4}{0}(\weyl_\lambda^k)\simeq \fHdeg{2}{0}(\fHdeg{1,3}{0}(\weyl_\lambda^k)),&& \fHdeg{1,3}{0}(\weyl_\lambda^k)\simeq \fHdeg{1,2}{0}(\fHdeg{1,1,2}{0}(\weyl_\lambda^k)).
\end{array}
\end{align*}
\end{ThmLetter} 
The proof is again based on the free field realizations of these modules available at generic levels.

A direct consequence of Theorem~\ref{thm:main_isom} is that the one-step Bershadsky--Polyakov type reduction here coincides with two Virasoro type reductions applied to the minimal $\W$-algebra $\pW{1,1,2}$ in $\sll_4$.
More interestingly, this result between one-step and two-step reductions can be upgraded at the level of Weyl modules as emphasized in Theorem~\ref{mainthm:modules}:
setting $M=\fHdeg{1,1,2}{0}(\weyl_\lambda^k)$,
\begin{align*}
    \fHdeg{1,2}{0}(M)\simeq \fHdeg{2}{0}(\fHdeg{2}{0}(M))
\end{align*}
holds as an isomorphism of $\pW{1,3}$-modules.
This implies the compatibility of the functors for the two multi-steps reductions.

\subsection*{Organization of the paper} 
The first two sections provide the necessary background: Section \ref{sec: preliminary} starts with a quick overview of the $\W$-algebras associated with $\sll_4$ whereas Section \ref{sec:ffr} presents their free field realizations and introduces the screening operators which consist in one main ingredient of the proof of Theorem~\ref{thm:main_isom}.
The next two sections focus on the partial reductions.
More precisely, Section \ref{Sec: Partial reductions} shows the isomorphism for the $0$-th cohomology at generic levels. 
The generalization to all levels and the cohomology vanishing for other degrees is established in Section \ref{sec:cohomology_vanishing}.
The dual to partial reductions, i.e. the inverse Hamiltonian reductions, is studied in Section \ref{sec:ihr}.
Finally, Section \ref{sec: Partial reductions for modules} is devoted to the proof of Theorem~\ref{mainthm:modules}.

\vspace{1em}
\paragraph{\textbf{Acknowledgements}} 
We thank Thibault Juillard and Dylan Butson for interesting communication.
S.N. would like to thank Gurbir Dhillon for useful discussions.
J.F. is supported by a University of Melbourne Establishment Grant. 
Z.F. is supported by the Australian Research Council Discovery Project DP210101502.
E.F. is supported by an Australian Government Research Training Program (RTP) Scholarship.

\section{Preliminary} \label{sec: preliminary}
Let $\g$ be a finite-dimensional Lie algebra over $\C$ and denote by $\widehat{\g}$ the corresponding affine Kac--Moody algebra.
The affine vertex algebra $\V^\kk(\g)$ associated with $\g$ at level $\kk\in\C$ is the parabolic Verma $\widehat{\g}$-module
\begin{equation}
    \V^\kk(\g)=U(\widehat{\g})\otimes_{U(\g[t]\oplus\C K)}\C_\kk
\end{equation}
where $\C_\kk$ is the one-dimensional representation of $\g[t]\oplus\C K$ on which $\g[t]$ acts trivially and $K$ acts as the multiplication by the scalar $\kk$.
There is a unique vertex algebra structure on $\V^\kk(\g)$, which is generated by the fields
\begin{equation}
    x(z)=\sum_{n\in\Z}x_n z^{-n-1},\qquad x_n=x\otimes t^n,\, x\in\g
\end{equation}
which satisfy the OPEs
\begin{equation}
    x(z)y(w)\sim \frac{[x,y](w)}{(z-w)}+\frac{(x,y)\kk \wun}{(z-w)^2},\quad x,y\in\g
\end{equation}
where $(\cdot, \cdot)$ is a normalized invariant inner product of $\g$ proportional to the Killing form.

The affine vertex algebra $\V^\ell(\h)$ corresponding to a commutative Lie algebra $\h\simeq(\gl_1)^{\oplus n}$ (such as the Cartan subalgebra of $\g$) is referred to as the rank $n$ Heisenberg vertex algebra and will often be denoted by $\pi_\h^\ell$.

$\W$-algebras are vertex algebras obtained as the quantum Hamiltonian reductions of affine vertex algebras \cite{FF90,KRW03}. 
More precisely, we take a (non-zero) nilpotent element $f$ and a \emph{good} $\Z$-grading\footnote{In general, the grading is a $\frac{1}{2}\Z$-grading, but in type $A$, we can always find a $\Z$-grading. See \cite{BG05,EK05} for the classification.}
$\Gamma\colon \g=\bigoplus_j\g_j$ with respect to $f$. 
Let $\bigwedge{}^{\bullet}_{\varphi,\varphi^*}$ denote the $bc$-system strongly and freely generated by the odd fields $\varphi(z),\varphi^*(z)$ satisfying the OPEs
\begin{align}
 \varphi(z)\varphi^*(w)\sim \frac{\wun}{(z-w)},\quad  \varphi(z)\varphi(w)\sim 0\sim \varphi^*(z)\varphi^*(w).
\end{align}
By taking the nilpotent subalgebra $\g_+:=\bigoplus_{j>0}\g_j$, we introduce a BRST complex 
\begin{align}\label{BRST cohomology}
    C_f^\bullet(\V^\kk(\g))=\V^\kk(\g)\otimes \bigwedge{}^{\bullet}_{\g_+},\quad
    \bigwedge{}^{\bullet}_{\g_+}:=\bigotimes_{\begin{subarray}{c}\alpha \in \Delta_{+}\\ \Gamma(\alpha)>0\end{subarray}} \bigwedge{}^{\bullet}_{\varphi_\alpha,\varphi^*_\alpha}
\end{align}
equipped with the differential 
\begin{align}
    d=\int Y(Q,z)\ \dd z,\quad Q=Q_{\mathrm{st}}+Q_\chi
\end{align}
with 
\begin{align}
    Q_{\mathrm{st}}=\sum e_\alpha \varphi_\alpha^*-\frac{1}{2}\sum c_{\alpha,\beta}^\gamma \varphi_\alpha^*\varphi_\beta^* \varphi_\gamma,\quad Q_\chi=\sum (f,e_\alpha)\varphi_\alpha^*
\end{align}
where the sums are taken over positive roots $\alpha$'s for $\g_+$ (i.e.\ $ \Gamma(\alpha)>0$)
and $c_{\alpha,\beta}^\gamma$ are the structure constants defined by the relations $[e_\alpha,e_\beta]=\sum c_{\alpha,\beta}^\gamma e_\gamma$.

The $\W$-algebra $\W^\kk(\g,f)$ associated with $(\g,f)$ at level $\kk$ is defined as 
\begin{align}
    \W^\kk(\g,f):=\HH^0(C_f^\bullet(\V^\kk(\g)),d).
\end{align}
Note that $\W^\kk(\g,f)$ is independent of the choices of representatives $f$ in a given nilpotent orbit and good gradings $\Gamma$ \cite{AKM15} so that we will occasionally use the notation $\W^\kk(\OO_f)$ in the following.
Moreover, we have the cohomology vanishing $\HH^{\neq0}(C_f^\bullet(\V^\kk(\g)),d)=0$ \cite{KW04}.

Given a $\V^\kk(\g)$-module $M$, we have a BRST complex $C_f^\bullet(M)$ defined as in \eqref{BRST cohomology}, whose cohomology $\HH_{\OO_f}(M):=\HH(C_f^\bullet(\V^\kk(\g)),d)$ is a module over $\W^\kk(\OO_f)$.

In this paper, we realize $\sll_4$ as the set of $4\times 4$-matrices with zero-trace. 
A convenient basis is given by the elementary matrices $\EM{i,j}$ ($1\leq i\neq j\leq 4$) together with the diagonal matrices $\EM{i,i}-\EM{i+1,i+1}$ ($1\leq i\leq 3$).
The latter matrices form a basis for a Cartan subalgebra $\h$ of $\sll_4$.
Denote $\Pi=\{\alpha_1,\alpha_2,\alpha_{3}\}$ the set of simple root of $\sll_4$.
Then the set of positive roots $\Delta_+$ is given by
\begin{equation}
    \alpha_1,\quad\alpha_2,\quad\alpha_3,\quad
    \alpha_4=\alpha_1+\alpha_2,\quad
    \alpha_5=\alpha_2+\alpha_3,\quad
    \alpha_6=\alpha_1+\alpha_2+\alpha_3.
\end{equation}
The corresponding root vectors and coroots are given by
\begin{equation}
\begin{gathered}
    e_{i}=\EM{i,i+1},\quad 
    h_i=\EM{i,i}-\EM{i+1,i+1},\quad
    f_{i}=\EM{i+1,i},\quad(i=1,2,3)\\
    e_4=\EM{1,3},\quad e_5=\EM{2,4},\quad e_6=\EM{1,4},\\
    f_4=\EM{3,1},\quad f_5=\EM{4,2},\quad f_6=\EM{4,1}.
\end{gathered}
\end{equation}
We denote by $W$ the Weyl group of $\sll_4$ (which is generated by the simple reflections $s_i$, $i =1,2,3$), by $\ell\colon W\rightarrow \Z_+$ the length function such that the longest element of $W$ has length $N=6$, and by $\circ\colon W\times \h^*\rightarrow \h^*$ the dot action given by $w\circ\lambda=w(\lambda+\rho)-\rho$.

In type $A$, nilpotent orbits are parametrized by partition of integers.
In particular, the Lie algebra $\sll_4$ admits five distinct nilpotent orbits: trivial, minimal, rectangular, subregular and regular corresponding to the partitions $(1^4)$, $(1,1,2)$, $(2,2)$, $(1,3)$ and $(4)$ respectively.
For a partition $\lambda$, we denote by $\mathbb{O}_\lambda$ the corresponding orbit inside $\sll_4$. 

The nilpotent orbits are partially ordered by closure relations for the Zariski topology, that is $\mathbb{O}_\lambda\leq \mathbb{O}_\mu$ iff ${\mathbb{O}}_\lambda\subset \overline{\mathbb{O}}_\mu$. This also coincides with the dominance order defined on the set of partitions. 
In Fig.\ \ref{fig:partial HR_sl4}, the nilpotent orbits are represented by the Young tableaux of the corresponding partition and the relations between them are indicated by arrows;
the left-most side is the one corresponding to the smallest nilpotent orbit, i.e. $\mathbb{O}_{(1^4)}=\{0\}$, and the right-most side is the one corresponding to the principal nilpotent orbit, i.e. $\overline{\mathbb{O}}_{(4)}$ is the nilpotent cone of $\sll_4$.

We end this section by introducing the $\beta\gamma$-system and the half-lattice vertex algebra $\Pi$, both of which appear frequently in what follows.
The $\beta\gamma$-system is the vertex algebra strongly and freely generated by the even fields $\beta(z),\gamma(z)$ satisfying the OPEs
\begin{equation}
    \beta(z)\gamma(w)\sim \frac{\wun}{(z-w)},\quad  \beta(z)\beta(w),  \gamma(z)\gamma(w)\sim 0
\end{equation}
whereas the half-lattice vertex algebra $\Pi$
is a subalgebra of the lattice vertex superalgebra $V_{\Z u\oplus \Z v}$ associated with the lattice $\Z u\oplus \Z v$ equipped with bilinear form defined by $(u,u)=1=-(v,v)$:
\begin{align}
  \Pi:=\bigoplus_{n\in \Z}\pi^{u,v}_{n(u+v)}\subset V_{\Z u\oplus \Z v}.  
\end{align}
It is useful to introduce the vectors
\begin{align}
    c=u+v,\quad d=u-v.
\end{align}
The $\bg$-system embeds into $\Pi$ using the FMS bosonisation \cite{FMS86}:
\begin{align}\label{FMS}
\begin{split}
&\beta\gamma \xrightarrow{\simeq} \mathrm{Ker} \left( S_{\mathrm{FMS}}\colon \Pi\rightarrow V_{\Z u\oplus \Z v} \right), \quad S_{\mathrm{FMS}}:=\int Y(e^u,z)\,\dd z,\\
&\beta\mapsto \hwt{c},\quad \gamma \mapsto -u\hwt{-c}.
\end{split}
\end{align}

\section{Free fields realizations of \texorpdfstring{$\W$}{W}-algebras}\label{sec:ffr}
In the rest of the paper, we assume that $\kk$ is a generic level (e.g.\ $\kk\in \C \backslash \Q$) unless otherwise stated.
Then the $\W$-algebras can be realized as the kernel of certain screening operators acting on free field algebras \cite{Gen20}. In this section, we recall this construction for later use.

Let $N_+\subset \SL_4$ be the upper unipotent subgroup, i.e.
\begin{equation}
   N_+=\left\{M(z); z_1,\ldots, z_6\in \C\right\}.
\end{equation}
with 
\begin{align}\label{Upper unipotents}
    M(z)=M(z_1,\ldots,z_6)=\left(\begin{array}{cccc}
    1&z_1&z_4&z_6\\
    0&1&z_2&z_5\\
    0&0&1&z_3\\
    0&0&0&1\\
    \end{array}\right).
\end{align}
The left action of $N_+$ on itself by matrix multiplication induces an anti-homomorphism $\rho:\n_+\to\D(N_+)$ where $\nil_+$ is the Lie algebra of $N_+$ and $\D(N_+)$ is the ring of polynomial differential operators on $N_+$. 
In particular, we have 
\begin{equation}\label{eq:rho}
    \rho(e_i)=\sum_{j=1}^6P_i^{j}(z)\partial_{z_j}\qquad (i=1,2,3) 
\end{equation}
for some polynomials $P_i^{j}(z)\in\C[z_1,\dots,z_6]$.
The affine vertex algebra $\V^\kk(\sll_4)$ admits an embedding into the free field algebra
\begin{align}\label{Wakimoto realization for affine}
   \Psi\colon \V^\kk(\sll_4)\hookrightarrow \ff{1,\ldots,6}\otimes \heis,
\end{align}
called the Wakimoto realization (see e.g.\ \cite{Fre07}) by using the tensor product of the $\beta\gamma$-systems $\ff{1,\ldots,6}$ and the Heisenberg vertex algebra $\heis$ associated with the Cartan subalgebra $\h$ at level $\kk+4$. 

The image of $\Psi$ is equal to the intersection of the kernels of the screening operators 
\begin{align}\label{affine screenings}
    S_i=\int S_i(z)\ \dd z \colon \affWak{0} \rightarrow \affWak{-\alpha_i},\quad S_i(z)= Y(\widehat{\rho}(e_i)\fockIO{i},z)
\end{align}
between the Wakimoto representations $\affWak{\lambda}:=\ff{1,\ldots,6}\otimes \Fock{\lambda}$, which are $\V^\kk(\sll_4)$-modules through \eqref{Wakimoto realization for affine}. Here $\widehat{\rho}(e_i)$ is the element in $\ff{1,\ldots,6}$ obtained from \eqref{eq:rho} by replacing $z_j$ with $\gamma_j$ and $\partial_{z_j}$ with $\beta_j$ (and taking normally-ordered products where necessary) and $\Fock{\lambda}$ denotes the Fock module of $\heis$ of highest weight $\lambda\in\h^*$.

The embedding \eqref{Wakimoto realization for affine} is completed into an exact sequence
\begin{align}\label{Wakimoto resolution of affine}
    0\rightarrow  \V^\kk(\sll_4) \overset{\Psi}{\longrightarrow} \affWak{0}\overset{\bigoplus S_i}{\longrightarrow} \bigoplus_{i=1,\ldots,3} \affWak{-\alpha_i}\rightarrow \mathrm{C}_2 \rightarrow \cdots \rightarrow \mathrm{C}_N \rightarrow 0
\end{align}
where
\begin{align}
    \mathrm{C}_i=\bigoplus_{\begin{subarray}c w\in W\\ \ell(w)=i\end{subarray}} \affWak{w\circ 0},
\end{align}
see e.g.\ \cite[Proposition 4.2]{ACL19}.

The Wakimoto realizations of $\W$-algebras are obtained by applying $\HH_{\OO_f}^0$ to \eqref{Wakimoto resolution of affine}.
Since the $\V^\kk(\sll_4)$-modules $\affWak{\lambda}$ are $\HH_{\OO_f}$-acyclic (see e.g. \cite[Proposition 4.5]{Gen20}), we obtain the complex 
\begin{align}\label{Wakimoto resolution of W algebras}
    0\rightarrow  \W^\kk(\OO_f) \rightarrow \HH_{\OO_f}^0(\affWak{0})\overset{\bigoplus [S_i]}{\longrightarrow} \bigoplus_{i=1,\ldots,3} \HH_{\OO_f}^0(\affWak{-\alpha_i})\rightarrow \cdots \rightarrow \HH_{\OO_f}^0(\mathrm{C}_N) \rightarrow 0,
\end{align}
which is exact by the cohomology vanishing $\HH_{\OO_f}^{\neq0}(\V^\kk(\sll_4))=0$.
The $\W^\kk(\OO_f)$-modules $\HH_{\OO_f}^0(\affWak{\lambda})$ and the induced screening operators $[S_i]$ are explicitly obtained by using good coordinates of \eqref{Upper unipotents} as follows.

Consider the Lie subgroups $G_+$ and $G_{0,+}$ of $N_+$ which correspond to the subalgebras $\g_+$ and $\g_{0,+}:=\g_0\cap \nil_+$ respectively under the exponential map $\exp:\nil_+\to N_+$.
Let us specialize the matrix $M(z)$ in \eqref{Upper unipotents} as 
\begin{equation}\label{some matrices}
\begin{aligned}
    M_+(z)=M(z)|_{\begin{subarray}{l} z_i=0\\ \mathrm{s.t.} \Gamma(\alpha_i)=0\end{subarray}},\qquad 
    M_0(z)=M(z)|_{\begin{subarray}{l} z_i=0\\ \mathrm{s.t.} \Gamma(\alpha_i)>0\end{subarray}}
\end{aligned}
\end{equation}
and introduce the following subgroups
\begin{align*}
 &G_{0,+}^{(1)}=\{M_0^{(1)}(z):=M_0(z_1,0,\dots,0)\},\\
 &G_{0,+}^{(2)}=\{M_0^{(2)}(z):=M_0(0,z_2,0,z_4,0,0)\},\\
 &G_{0,+}^{(3)}=\{M_0^{(3)}(z):=M_0(0,0,z_3,0,z_5,z_6)\}.
\end{align*}
Note that the variables here might be zero due to the constraints \eqref{some matrices}. Then we have
\begin{equation}
 \left(G_{0,+}^{(1)}\times G_{0,+}^{(2)}\times G_{0,+}^{(3)}\right)\ltimes G_+ \xrightarrow{\simeq} N_+
\end{equation}
as varieties. We take new coordinates $(z_1,\ldots,z_6)$ on $N_+$ through the product
\begin{align}
    M_0^{(1)}(z)M_0^{(2)}(z)M_0^{(3)}(z)M_+(z).
\end{align}
Note that the simplicity of the decomposition is due to the small rank of the Lie algebra $\sll_4$.
Then by \cite[Proposition 4.5, Theorem 4.8]{Gen20}, we have
\begin{align}\label{Wakimoto for Walg}
    \HH_{\OO_f}^0(\affWak{\lambda})\simeq \bg^\bigstar\otimes \Fock{\lambda},\quad  \bigstar=\{i=1,\ldots,6 ; \Gamma(\alpha_i)=0\}
\end{align}
and the induced screening operators $[S_i]$ are expressed as
\begin{equation}\label{screeinings for Walg}
    S_i^f=\int S_i^f(z)\ \dd z,\quad S_i^f(z):=Y(\widehat{\rho}_{f,\Gamma}(e_i)\fockIO{i},z)
\end{equation}
where 
\begin{align}
    \widehat{\rho}_{f,\Gamma}(e_i)=\begin{cases}
         \widehat{\rho}(e_i)& \text{if }\Gamma(\alpha_i)=0, \\
         \displaystyle{\sum_{j=1}^6(f,e_j)P_i^{j}(\gamma)}& \text{if }\Gamma(\alpha_i)=1.
    \end{cases}
\end{align}
Roughly, they are obtained from $S_i$ by evaluating $\beta_j$ as $(f,e_j)$ if $\alpha_j$ is positively graded.
In particular, the $\W$-algebras have the following free field realization 
\begin{align}\label{Wakimoto realizations for W-alg}
    \W^\kk(\OO_f)\simeq \bigcap_{i=1,\ldots,3} \ker  S_i^f|_{\bg^\bigstar\otimes \heis}.
\end{align}

We will frequently use the notation
\begin{align}
\affWak{\lambda,f}:=\bg^\bigstar\otimes \Fock{\lambda}
\end{align}
for the Wakimoto representations over $\W^\kk(\OO_f)$ and replace nilpotent elements with their corresponding Young tableaux. 
For example, 
\begin{align}
  \W^\kk(\sll_4,f_{(2,2)})=\pW{2,2}=\fHdeg{2,2}{0}(\V^\kk(\sll_4)),\quad \affWak{\lambda,f_{(2,2)}}=\pWak{2,2}{\lambda}.
\end{align}

The equivalence classes of pairs $(f,\Gamma)$ are classified by pyramids described by Young tableaux with shifting \cite{EK05}.
For example, the three following pyramids correspond to the nilpotent orbit $\mathbb{O}_{(2,1,1)}$: 
\begin{equation}\label{shifting Young diagrams}
    {
    \begin{tikzpicture}[every node/.style={draw,regular polygon sides=4,minimum size=0.4cm,line width=0.04em},scale=0.4]
        \node at (0,0) (4) {};
        \node at (1,0) (1) {};
        \node at (0,1) (4) {};
        \node at (0,2) (1) {};
    \end{tikzpicture}}\qquad
        {
    \begin{tikzpicture}[every node/.style={draw,regular polygon sides=4,minimum size=0.4cm,line width=0.04em},scale=0.4]
        \node at (0,0) (4) {};
        \node at (1,0) (1) {};
        \node at (0.5,1) (4) {};
        \node at (0.5,2) (1) {};
    \end{tikzpicture}} \qquad
            {
    \begin{tikzpicture}[every node/.style={draw,regular polygon sides=4,minimum size=0.4cm,line width=0.04em},scale=0.4]
        \node at (0,0) (4) {};
        \node at (1,0) (1) {};
        \node at (1,1) (4) {};
        \node at (1,2) (1) {};
    \end{tikzpicture}.}
\end{equation}
The explicit pair $(f,\Gamma)$ associated with a pyramid is obtained by labelling each box by $\{1,\ldots,4\}$ (for $\sll_4$) corresponding to the basis of the natural representation $\C^4=\bigoplus_{i=1}^4\C v_i$. Then, the nilpotent element $f$ is expressed as 
\begin{equation}
    f=\sum_{i,j=1}^4\delta_{i\rightarrow j}E_{i,j}
\end{equation}
where $\delta_{i\rightarrow j}$ is $1$ if the boxes {\tiny $\numtableaux{i}$} and {\tiny $\numtableaux{j}$} are adjacent and $0$ otherwise. 
The labelled pyramid gives a grading on $v_1,\ldots, v_4$ by reading the $x$-coordinates of the centers of the corresponding boxes with the left-most boxes satisfying $x=0$ and a decreasing order to the right.
It induced a grading on $\C^4$ and then $\sll_4\subset \mathrm{End}\ {\C^4}$. For example, the pyramid corresponding to the minimal nilpotent orbit 
\begin{equation}
\numtableaux{3\\2\\1&4}
\end{equation}
gives the grading on $\C^4$, $\wt(v_1) =0,\ \wt(v_2)=0,\ \wt(v_3)=0,\ \wt(v_4)=-1$, and defines the pair $(f,\Gamma)$:
\begin{equation}
f=    \begin{pmatrix}
0 & 0 & 0& 0\\ 0 & 0 & 0& 0\\ 0 & 0 & 0& 0\\ 1 & 0 & 0 & 0
    \end{pmatrix} ,\qquad \Gamma\colon     \begin{pmatrix}
0 & 0 & 0& 1\\ 0 & 0 & 0& 1\\ 0 & 0 & 0& 1\\ -1 & -1 & -1 & 0
    \end{pmatrix}. \\
\end{equation}
Note that different choices of the labelled pyramids might result in different sets of screening operators which all satisfy \eqref{Wakimoto realizations for W-alg}.
\begin{proposition}\label{Several Wakimoto realizations for W-algebras}
The data for the realizations \eqref{Wakimoto realizations for W-alg} of the $\W$-algebras associated with the labelled Young tableaux
\begin{equation*}
    \numtableaux{3\\2\\1&4}\quad   
    \numtableaux{\none&4\\\none&3\\1&2}\quad   
    \numtableaux{2&4\\1&3}\quad    
    \numtableaux{2&3\\1&4}\quad    
    \numtableaux{\none&\none&4\\1&2&3}\quad    
    \numtableaux{\none&3\\1&2&4}\quad  
    \numtableaux{1&2&3&4}         
\end{equation*}
 is explicitly given in Table \ref{tab: list of Wakimoto realizations of Walg} below. 
\end{proposition} 

{\small
\begin{table}[h!]
    \begin{center}
    \begin{tabular}{|c|>{\centering\arraybackslash} m{2cm}|c|c|c|c|} \hline
        Case & Pyramid & $\bigstar$ & $\widehat{\rho}_{f,\Gamma}(e_1)$ & $\widehat{\rho}_{f,\Gamma}(e_2)$ & $\widehat{\rho}_{f,\Gamma}(e_3)$ \\ \hline
         (I.1) & \vspace{1mm} $\numtableaux{3\\2\\1&4}$ &$\{1,2,4\}$  &$\coffaa$ & $\coffab$& $\coffac$ \\ \hline
         (I.2) & \vspace{1mm} $\numtableaux{\none&4\\\none&3\\1&2}$ &$\{2,3,5\}$ &$\coffba$ & $\coffbb$& $\coffbc$ \\ \hline
         (II.1) & \vspace{1mm} $\numtableaux{2&4\\1&3}$ &$\{1,3\}$ &$\coffca$ & $\coffcb$& $\coffcc$ \\ \hline
         (II.2) & \vspace{1mm} $\numtableaux{2&3\\1&4}$ &$\{1,3\}$ &$\coffda$ & $\coffdb$& $\coffdc$ \\ \hline
         (III.1) & \vspace{1mm} $\numtableaux{\none&\none&4\\1&2&3}$ &$\{3\}$ &$\coffea$ & $\coffeb$& $\coffec$ \\ \hline
         (III.2) & \vspace{1mm} $\numtableaux{\none&3\\1&2&4}$ &$\{2\}$ &$\cofffa$ & $\cofffb$& $\cofffc$ \\ \hline
         (IV) & \vspace{1mm} $\numtableaux{1&2&3&4}$ &$\emptyset$ &$\coffga$ & $\coffgb$& $\coffgc$ \\ \hline
    \end{tabular}
    \end{center}
    \caption{Coefficients $\widehat{\rho}_{f,\Gamma}(e_i)$ in the screening operators $S_i^f$.}
    \label{tab: list of Wakimoto realizations of Walg}
\end{table}}

Later, we will also use the following choice of subregular nilpotent element and good grading: 
\begin{equation}\label{inhomogeneous subregular}
    f_{\mathrm{sub}}^{(1)}=\EM{2,1}+\EM{3,1}+\EM{4,2}+\EM{4,3},\qquad 
    \Gamma\colon \begin{pmatrix}
       0&1&1&2\\
    -1&0&0&1\\
    -1&0&0&1\\
    -2&-1&-1&0
    \end{pmatrix}
\end{equation}
These data can be encoded, following the same rules as before, in the following unconventional tableau
\begin{equation}
    {
    \begin{tikzpicture}[every node/.style={draw,regular polygon sides=4,minimum size=0.4cm,line width=0.04em},scale=0.4]
        \node at (1,0) (4) {\small 4};
        \node at (-1,0) (1) {\small 1};
        \node at (0,.54) (2) {\small 2};
        \node at (0,-.54) (3) {\small 3};
    \end{tikzpicture}.}
\end{equation}
Indeed, $f_{\mathrm{sub}}^{(1)}$ is conjugate in $\SL_4$ to the one $f_{\mathrm{sub}}^{(2)}=E_{4,2}+E_{2,1}$ in Table \ref{tab: list of Wakimoto realizations of Walg} (III.2) as 
\begin{equation}
   f_{\mathrm{sub}}^{(1)}=gf_{\mathrm{sub}}^{(2)}g^{-1},\qquad
    g=\begin{pmatrix}
        1/2&0&0&0\\
        0&1/2&-2&0\\
        0&1/2&2&0\\
        0&0&0&1
    \end{pmatrix}
\end{equation}
so that $g$ preserves the common good grading.

\begin{proposition}\label{inhomogeneous screening for sub}
For the $\W$-algebra $\W^\kk(\OO_f)$ associated with $f=f_{\mathrm{sub}}^{(1)}$ satisfying \eqref{inhomogeneous subregular}
the realization \eqref{Wakimoto realizations for W-alg} is given by 
\begin{align}\label{eq:screening_sub1}
    \widehat{\rho}_{f,\Gamma}(e_1)=\wun+\gamma_2,\quad\widehat{\rho}_{f,\Gamma}(e_2)=\beta_2,\quad \widehat{\rho}_{f,\Gamma}(e_3)=\wun-\gamma_2.
\end{align}
\end{proposition}

\section{Partial reductions}\label{Sec: Partial reductions}
The $\W$-algebras $\W^k(\OO_f)$ are strongly and freely generated by fields corresponding to a basis of the centralizer $\g^f$.  In type $A$, the strong generating type of $\W^k(\OO_\lambda)$ is described in Proposition 2.1 of \cite{CFLN}, in addition to the OPEs for $n\leq 5$. We will use the same generators here, although the detailed OPEs are unnecessary.

The minimal $\W$-algebra $\pW{1,1,2}$ is strongly and freely generated by 
\begin{align}\label{strong generators for the minimal}
    e,h,f,J,\quad v_1^{\pm}, v_2^{\pm},\quad  L
\end{align}
of conformal weight $1$, $3/2$ and $2$ respectively. The first four generate the affine vertex algebra $\V^{\kk+1}(\gl_2):=\V^{\kk+1}(\sll_2)\otimes \pi^J$ with 
\begin{align}
    J(z)J(w)\sim \frac{(\kk+2)}{(z-w)^2}.
\end{align}
The second four elements span the natural representation of $\gl_2$ and its dual 
\begin{align}
    \C^2=\Span\{v_1^+,v_2^+\},\quad \overline{\C}^2=\Span\{v_1^-,v_2^-\}.
\end{align} and satisfy the OPEs
\begin{align}\label{OPEs for the natural representations}
    v^+_i(z) v^+_j(w)\sim 0,\quad v^-_i(z) v^-_j(w)\sim 0,\quad (i,j=1,2).
\end{align}
The last element $L$ is a conformal vector of $\pW{1,1,2}$. As $\pW{1,1,2}$ contains $\V^{\kk+1}(\sll_2)$, we can take the quantum Hamiltonian reduction $\fH{2}$, which reduces $\V^{\kk+1}(\sll_2)$ to the Virasoro vertex algebra $\pWlevel{2}{\kk+1}$.
More explicitly, we consider the cohomology of the following BRST complex
\begin{align}\label{BRST complex from min to subreg}
    \pW{1,1,2}\otimes \bigwedge{}^{\bullet}_{\varphi,\varphi^*},\quad \fd{2}=\intNOz{(e+\wun) \varphi^*}.
\end{align}

On the other hand, the basis $\{v_1^+, v_2^+\}$ of $\C^2$ equipped with the $\gl_2$-action can be identified with the root vectors $E_{1,2}, E_{1,3}$ inside $\sll_3$. They satisfy the same (vanishing) OPEs thanks to \eqref{OPEs for the natural representations}. 
Therefore, we may apply the quantum Hamiltonian reduction $\fH{1,2}$.
More explicitly, we consider the following BRST complex
\begin{align}\label{second BRST complex 1}
    \pW{1,1,2}\otimes \bigwedge{}^{\bullet}_{\varphi_{1,2},\varphi_{1,2}^*},\quad  \fd{1,2}=\intNOz{(v_1^++\wun)\varphi_1^*+v_2^+\varphi_2^*}.
\end{align}
The next theorem identifies the resulting vertex algebras with $\W$-algebras associated with nilpotent orbits larger than the original ones in the closure relations.

\begin{theorem}\label{thm:min_to_rect}
For generic levels $\kk$, there exist isomorphisms of vertex algebras
\begin{align}\label{eq:partial_min}
    \fHdeg{2}{0}(\pW{1,1,2})\simeq \pW{2,2},\qquad \fHdeg{1,2}{0}(\pW{1,1,2})\simeq \pW{1,3}.
\end{align}
\end{theorem}
\begin{proof}
We show the first isomorphism. 
We apply $\fH{2}$ to the exact sequence \eqref{Wakimoto resolution of W algebras} for Table \ref{tab: list of Wakimoto realizations of Walg} (I.1).
In \eqref{Wakimoto realizations for W-alg}, the affine subalgebra $\V^{\kk+1}(\sll_2)$ is realized as 
    \begin{align}
    \begin{split}
        &e=\beta_2,\qquad h=\beta_1\gamma_1-2\beta_2\gamma_2-\beta_4\gamma_4+b_2,\\
        &f=\beta_1\gamma_1\gamma_2 + \beta_1\gamma_4 - \beta_2\gamma_2^2 - \beta_4\gamma_2\gamma_4 + \gamma_2 b_2 + (\kk+1) \partial\gamma_2.
    \end{split}
    \end{align}
Then the BRST complex for the Wakimoto representation $\pBRST{2}(\pWak{1,1,2}{\lambda})$ is 
\begin{align}
    \pBRST{2}(\pWak{1,1,2}{\lambda})=\pWak{1,1,2}{\lambda}\otimes  \bigwedge{}^{\bullet}_{\varphi,\varphi^*},\quad \fd{2}=\intNOz{(\beta_2+\wun) \varphi^*}
\end{align}
and thus 
\begin{align}
    \fHdeg{2}{n}(\pWak{1,1,2}{\lambda})\simeq \delta_{n,0} \ff{1,4}\otimes \Fock{\lambda}.
\end{align}
Therefore, $\fH{2}(\pW{1,1,2})$ is computed by the cohomology of the complex 
\begin{align}
0\rightarrow \fHdeg{2}{0}(\pWak{1,1,2}{0}) \overset{[d_0]}{\rightarrow} \fHdeg{2}{0}(\presol{1,1,2}{1}) \overset{[d_1]}{\rightarrow} \cdots \overset{[d_{N-1}]}{\rightarrow} \fHdeg{2}{0}(\presol{1,1,2}{N}) \rightarrow 0
\end{align}
and thus 
\begin{align}\label{FFR by two-step reduction 1}
\fHdeg{2}{0}(\pW{1,1,2})\simeq \bigcap_{i=1}^3 \ker {\ts{i}{1,1,2}}|_{\ff{1,4}\otimes \heis},\quad \ts{i}{1,1,2}=\scrNO{i}{P_i}
\end{align}
where 
\begin{align}\label{coefficients in the second reduction}
    P_1=\beta_1,\quad P_2=-(\wun+\gamma_1\beta_4),\quad P_3=\gamma_4.
\end{align}
By applying an isomorphism of $\bg$-systems we call a \emph{partial Fourier transform}
\begin{align}
    \iota:\ff{1,4}\xrightarrow{\simeq}\ff{1,3},\quad \beta_1,\gamma_1\mapsto\beta_1,\gamma_1,\qquad\beta_4,\gamma_4\mapsto-\gamma_3,\beta_3,
\end{align}
we find that the realization \eqref{FFR by two-step reduction 1} is equivalent to the one for $\pW{2,2}$ in Table \ref{tab: list of Wakimoto realizations of Walg} (II.2). This proves the first isomorphism.

The second isomorphism is shown similarly. We start with the resolution \eqref{Wakimoto resolution of W algebras} of $\pW{1,1,2}$ with Table \ref{tab: list of Wakimoto realizations of Walg} (I.2).   
Under this realization, $v_1^+, v_2^+$ in the differential $\fd{1,2}$ (see \eqref{second BRST complex 1}) are expressed as 
    \begin{align}
        v_1^+=\beta_2-\beta_5\gamma_3,\quad  v_2^+=-\beta_5\ (=e_{(-1)}v_1^+).
    \end{align}
    where $e=\beta_3$. 
    Hence, for the representations $\pWak{1,1,2}{\lambda}$, we have
    \begin{equation}
    \fd{1,2}=\intNOz{(\beta_2-\beta_5\gamma_3+\wun)\varphi_1^*-\beta_5\varphi_2^*}.
    \end{equation}
   After applying the automorphism
    \begin{align}
    \begin{array}{lll}
         \beta_1\mapsto \beta_2-\beta_5\gamma_3, & \beta_3\mapsto \beta_3-\beta_5\gamma_2, & \beta_5\mapsto \beta_5,\\
    \gamma_1\mapsto \gamma_2, & \gamma_3\mapsto \gamma_3, &\gamma_5 \mapsto \gamma_5+\gamma_2\gamma_3,   
    \end{array}  
    \end{align}
    on $\ff{2,3,5}$, $\fd{1,2}$ is expressed as 
        \begin{equation}
    \fd{1,2}=\intNOz{(\beta_2+\wun)\varphi_1^*+\beta_5\varphi_2^*}.
    \end{equation}
    Then the cohomology vanishing 
\begin{align}
    \fHdeg{2}{n}(\pWak{1,1,2}{\lambda})\simeq \delta_{n,0} \ff{3}\otimes \Fock{\lambda}
\end{align}
holds and the induced screening operators from Table \ref{tab: list of Wakimoto realizations of Walg} (II.2) are identified with those for $\pW{1,3}$ in Table \ref{tab: list of Wakimoto realizations of Walg} (III.1).
This proves the second isomorphism.
\end{proof}

Next, let us consider the rectangular $\W$-algebra $\pW{2,2}$ and the subregular $\pW{1,3}$.
The $\W$-algebra $\pW{2,2}$ is strongly and freely generated by 
\begin{align*}
    e,h,f,\qquad L_1,L_2, v_1^{\pm}
\end{align*}
of conformal weight $1$ and $2$ respectively. The first three generate the affine vertex algebra $\V^{2(\kk+2)}(\sll_2)$. The second two $L_1, L_2$ are mutually-commuting Virasoro vectors. The last two $v_1^{\pm}$ together with a linear combination of $L_1, L_2$ form an adjoint representation of $\sll_2$. 

Then we can take the quantum Hamiltonian reduction $\fH{2}$, which reduces $\V^{2(\kk+2)}(\sll_2)$ to the Virasoro vertex algebra $\pWlevel{2}{2(\kk+2)}$.
More explicitly, we consider the following BRST complex
\begin{align}\label{second differential for rec}
    \pW{2,2}\otimes \bigwedge{}^{\bullet}_{\varphi,\varphi^*},\quad \fd{2}=\intNOz{(e+\wun) \varphi^*}.
\end{align}

On the other hand, the $\W$-algebra $\pW{1,3}$ is strongly and freely generated by 
\begin{align*}
    J,\qquad L, v^\pm,\qquad \Omega_3
\end{align*}
of conformal weight $1$, $2$, $3$ respectively. The first generates the Heisenberg vertex algebra $\pi^J$ with
\begin{equation}
    J(z)J(w)\sim \frac{\frac{1}{4}(3\kk+8)}{(z-w)^2}.
\end{equation}
For the second, $L$ is a conformal vector, and $v^\pm$ are the highest weight vectors of Fock modules over $\pi^J$ of weight $\pm1$ which satisfy 
\begin{align}\label{OPEs for the natural representations for subreg}
    v^+(z)v^+(w)\sim 0,\quad v^-(z)v^-(w)\sim 0.
\end{align}
Finally, $\Omega_3$ is an element which commutes with $J$. 

The basis $v^+$ of $\C$ equipped with the $\gl_1$-action can be identified with the root vector $E_{1,2}$ inside $\sll_2$. Moreover, they satisfy the same (vanishing) OPEs. 
Therefore, we may apply the quantum Hamiltonian reduction $\fH{2}$.
More explicitly, we consider the following BRST complex
\begin{align}\label{second BRST complex 2}
    \pW{1,3}\otimes \bigwedge{}^{\bullet}_{\varphi,\varphi^*},\quad  \fd{2}=\intNOz{(v^++\wun)\varphi^*}.
\end{align}
The next theorem again identifies the resulting vertex algebras with $\W$-algebras associated with nilpotent orbits larger than the original ones in the closure relations.

\begin{theorem} \label{thm: pqhrRectoSub}
For generic levels $\kk$, there exist isomorphisms of vertex algebras
\begin{align}
    \fHdeg{2}{0}(\pW{2,2})\simeq \pW{1,3},\quad\fHdeg{2}{0}(\pW{1,3})\simeq \pW{4}.
\end{align}
\end{theorem}
\begin{proof}
We only sketch the proof as it is similar to that of Theorem \ref{thm:min_to_rect}.
By using the realization of $\pW{2,2}$ given in Table \ref{tab: list of Wakimoto realizations of Walg} (II.1), 
the affine subalgebra $\V^{\kk+1}(\sll_2)$ is realized as 
\begin{align}
    e=\beta_1+\beta_3,\quad h=\sum_{i=1,3}(-2\gamma_i\beta_i+h_i),\quad f=\sum_{i=1,3}(-2\gamma_i^2\beta_i+\gamma_ih_i+ k\partial\gamma_i).
\end{align}
Then the differential \eqref{second differential for rec} on the Wakimoto representations is 
\begin{align}
    \fd{2}=\intNOz{(\beta_1+\beta_3+\wun) \varphi^*}.
\end{align}
Using an automorphism
\begin{equation}\label{isomorphism of FFR}
\begin{gathered}
    \ff{\mathrm{I,II}}\xrightarrow{\simeq} \ff{1,3}\\
    \bt{I}  = \beta_1 - \beta_3, \quad 
    \gm{I} = \frac{1}{2}\big(\gamma_1 - \gamma_3 \big),\quad 
    \bt{II}  = \beta_1 + \beta_3, \quad
    \gm{II}= \frac{1}{2}\big(\gamma_1 + \gamma_3 \big),
\end{gathered}
\end{equation}
we obtain $\fHdeg{2}{n}(\pWak{2,2}{\lambda})\simeq \delta_{n,0}\ff{\mathrm{I}}\otimes \Fock{\lambda}$. 
Then, it maps the screening operators of $\pW{2,2}$ given in Table \ref{tab: list of Wakimoto realizations of Walg} (II.1) onto those \eqref{eq:screening_sub1} of $\pW{1,3}$. This proves the first isomorphism.

In the second case, by using the realization of $\pW{1,3}$ given in Table \ref{tab: list of Wakimoto realizations of Walg} (III.1), the differential is given by
\begin{align}
    \fd{2}=\intNOz{(\beta_3+\wun) \varphi^*}.
\end{align}
for the Wakimoto representations. Then we obtain $\fHdeg{2}{n}(\pWak{1,3}{\lambda})\simeq \delta_{n,0} \Fock{\lambda}$ and the realization of the induced screening operators of $\pW{1,3}$ given in Table \ref{tab: list of Wakimoto realizations of Walg} (III.1) onto those of $\pW{4}$ in Table \ref{tab: list of Wakimoto realizations of Walg} (IV). This proves the second isomorphism.
\end{proof}

\begin{remark} \hspace{0mm}
In Theorems \ref{thm:min_to_rect}, \ref{thm: pqhrRectoSub}, there is a simple rule relating the nilpotent orbits or, equivalently, the Young tableaux. Indeed, the resulting Young tableaux are obtained from the original ones by replacing the boxes corresponding to the ``affine part'' with those of their ``$\W$-algebras" as follows:
\begin{align}\label{Young tableau game}
\begin{array}{lclclcl}
\coloredtableaux{1,1,2}{1,1}&\overset{\fH{2}}{\longmapsto}&\coloredtableaux{2,2}{}, &&   \coloredtableaux{1,1,2}{1,1,1}&\overset{\fH{1,2}}{\mapsto} &   \coloredtableaux{1,3}{},\\
&&&&&&\\
\coloredtableaux{2,2}{1,1} &\overset{\fH{2}}{\mapsto} &   \coloredtableaux{1,3}{} , &&   \coloredtableaux{1,3}{1,1} & \overset{\fH{2}}{\mapsto} &   \coloredtableaux{4}{}.
\end{array}
\end{align}
    Here, we indicate the ``affine part'' by the shadow and use equivalences of Young tableaux by shifting; see \eqref{shifting Young diagrams}.
    Hence, the symbol of the reductions $\fH{2}$ and $\fH{1,2}$ is justified not only by the structure of BRST complexes, which are \emph{Virasoro type} and \emph{Bershadsky--Polyakov type} reductions
    \begin{align*}
    &\V^\kk(\sll_2)\mapsto \HH\left(\V^\kk(\sll_2)\otimes \bigwedge{}^{\bullet}_{\varphi,\varphi^*} \fd{2}\right),\ 
    &\V^\kk(\sll_3)\mapsto \HH\left(\V^\kk(\sll_3)\otimes \bigwedge{}^{\bullet}_{\varphi_{1,2},\varphi^*_{1,2}}, \fd{1,2}\right)
    \end{align*}
    with
    \begin{align*}
        \fd{2}=\intNOz{(E_{1,2}+\wun) \varphi^*},\quad \fd{1,2}=\intNOz{(E_{1,2}+\wun) \varphi^*_1+E_{1,3}\varphi^*_2}
    \end{align*}
    but also the combinatorial Young tableaux rule in \eqref{Young tableau game}.
\end{remark}

\begin{remark} \hspace{0mm}
In addition to those in \eqref{Young tableau game}, there is one more partial reduction: 
\begin{align}
\coloredtableaux{1,1,2}{1,1,1}\overset{\fH{3}}{\mapsto}   \coloredtableaux{4}{}.
\end{align}
Indeed, we have an embedding 
$$\V^0(\nil_+)\hookrightarrow \pW{1,1,2},\quad e, v_1^+,v_2^+\mapsto E_{2,3}, E_{1,2}, E_{1,3}$$
where $\nil_+\subset \sll_3$ is the upper nilpotent subalgebra.
Accordingly, one applies the partial reduction $\fHdeg{3}{0}(\pW{1,1,2})$. 
One can show that $\fHdeg{3}{0}(\pW{1,1,2})\simeq \pW{4}$ using Table \ref{tab: list of Wakimoto realizations of Walg} (I.2).
\end{remark}

As mentioned in the Introduction, the reductions by stages have been studied more in the context of \emph{finite $\W$-algebras}. 
They are defined from the enveloping algebras $\mathscr{U}(\g)$ of simple Lie algebras $\g$ through the finite analog of the BRST complex in \eqref{BRST cohomology}. 
Indeed, one associates to each (conformal) vertex algebra $V$ an associative algebra $A_V$, called \emph{the Zhu's algebra}. 
It is well known that $A_{\V^\kk(\g)}\simeq \mathscr{U}(\g)$. Moreover,
\begin{align}
    A_{C_f^\bullet(\V^\kk(\g))}\simeq \mathscr{U}(\g)\otimes \mathcal{C}\ell(\g_+\oplus \g_+^*)
\end{align}
with $\mathcal{C}\ell(\g_+\oplus \g_+^*)$ the Clifford algebra associated with $\g_+\oplus \g_+^*$. It has naturally a complex structure and agrees with the finite analog of the BRST complex.
By \cite{Ara07, DK06}, the cohomology $\mathscr{U}(\g,f)$, known as the finite $\W$-algebra, is isomorphic to the Zhu's algebra of the corresponding $\W$-algebra $\W^\kk(\g,f)$, i.e. 
\begin{align}
   \mathscr{U}(\g,f) \simeq A_{\W^\kk(\g,f)}.
\end{align}
Applying Zhu's functor $A_\bullet$ commutes with taking the (finite) BRST cohomology \cite[Theorem 8.1]{Ara07}. Therefore, when applying a partial reduction through the affine vertex subalgebras, we obtain the following statement immediately.
\begin{corollary}\label{finite W-algebras}
There exist isomorphisms of associative algebras
\begin{align}
    \fHdeg{2}{0}(\mathscr{U}(\sll_4,\ydiagram{1,1,2}))\simeq \mathscr{U}(\sll_4,\ydiagram{2,2}),\quad \fHdeg{2}{0}(\mathscr{U}(\sll_4,\ydiagram{2,2}))\simeq \mathscr{U}(\sll_4,\ydiagram{1,3}).
\end{align}
\end{corollary}
The second isomorphism was already obtained in \cite[\S 4.1]{GJ23} as well as 
the remaining cases 
\begin{align}
    \fHdeg{1,2}{0}(\mathscr{U}(\sll_4,\ydiagram{1,1,2}))\simeq \mathscr{U}(\sll_4,\ydiagram{1,3}),\quad \fHdeg{2}{0}(\mathscr{U}(\sll_4,\ydiagram{1,3}))\simeq \mathscr{U}(\sll_4,\ydiagram{4})
\end{align}
which we cannot recover directly from our results since we apply the reduction through a module over an affine vertex subalgebra rather than the subalgebra itself.
The first isomorphism seems to be new as far as we know, but the classical limit (i.e. after taking the associated graded of suitable filtrations) is obtained in \cite[Example 3.4.10]{Morgan14}.
We note that the first isomorphism completes the missing piece of the reduction by stages between finite $\W$-algebras for $\sll_4$ in \cite{GJ23}.

\section{Cohomology vanishing}\label{sec:cohomology_vanishing}
In this section, we prove cohomology vanishing of the partial Hamiltonian reductions upgrading the isomorphisms in Theorems \ref{thm:min_to_rect} and \ref{thm: pqhrRectoSub}. More explicitly, we have the following isomorphisms.

\begin{theorem}\label{Cohomology vanishing for all levels}
For all levels, we have the following cohomology vanishing 
\begin{align}
    &\fHdeg{1,2}{n}(\pW{1,1,2})\simeq \delta_{n,0}\pW{1,3},\qquad \label{eq:case1a}\\
    &\fHdeg{2}{n}(\pW{1,3})\simeq \delta_{n,0}\pW{4},\label{eq:case1b}\\
    &\fHdeg{2}{n}(\pW{1,1,2})\simeq \delta_{n,0}\pW{2,2},\qquad \label{eq:case2a}\\
    &\fHdeg{2}{n}(\pW{2,2})\simeq \delta_{n,0}\pW{1,3}.\label{eq:case2b}
\end{align}
\end{theorem}

Here, we shall distinguish two main cases. 
When the partial reduction functor is applied through an affine vertex subalgebra ($\V^\ell(\sll_2)$ in our cases) the cohomology vanishing follows from the cohomology vanishing of the modules in the Kazhdan--Lusztig category \cite{Ara05}.
This is the context of \eqref{eq:case2a}, \eqref{eq:case2b}.
Since Arakawa's result of cohomology vanishing holds for \emph{all levels}, these isomorphisms also hold for \emph{all levels}. The isomorphism \eqref{eq:case2a} was proven in \cite[Theorem A]{CFLN} using this argument.
Unfortunately, such a conceptual argument does not generalise to the other cases we consider in \S\ref{Sec: Partial reductions} for the moment. We will prove the cohomology vanishing of our specific examples by direct computation.
In the following, we prove the cases \eqref{eq:case1a}, \eqref{eq:case1b}.

\begin{proof}[Proof of \eqref{eq:case1a}]
We adapt the strategy of the proof of \cite[Theorem 4.1]{KW04} dealing with the cohomology vanishing for the reduction of the affine vertex algebra $\mathrm{H}_f^{n\neq0}(\V^\kk(\g))=0$ by using a certain decomposition of the BRST complex.
From \eqref{strong generators for the minimal} and \eqref{BRST complex from min to subreg}, the strong generators of the BRST complex 
\begin{equation}
\begin{gathered}
    C^{\bullet} = \pW{1,1,2}\otimes \bigwedge{}^{\bullet}_{\varphi_{1,2},\varphi_{1,2}^*},
    \quad
    \fd{1,2}=\intNOz{(v_1^++\wun)\varphi_1^*+v_2^+\varphi_2^*}
\end{gathered}
\end{equation}
can be chosen as
\begin{align}
\df{a},\ \df{b},\ \df{e},\ \df{f},\ \df{L},\ {v^+_i},\ {v_i^-},\ {\varphi^*_i},\ {\varphi_i} \quad ( i=1,2)
\end{align}
where we set 
\begin{equation}\label{sub-reg complex strong gens}
\begin{gathered}
\df{e}= e + \varphi_1 \varphi^*_2,\qquad \df{f} = f + \varphi_2\varphi^*_1,\qquad \df{a} = \frac{1}{2} \left( \df{h}+ \df{J} \right), \qquad \df{b} = \frac{1}{2} \left( \df{h} - \df{J} \right),\\
\df{L}= L + \frac{1}{2(\kk+3)} \left( \frac{1}{2} h^2 + ef + fe \right) + \frac{1}{2(\kk+2)}J^2 + \frac{1}{2} \partial h + \partial J + (\partial\varphi_1) \varphi^*_1 - \varphi_2 \partial\varphi^*_2
\end{gathered}
\end{equation}
with 
\begin{equation}
\df{h}= h + \varphi_1\varphi^*_1 - \varphi_2\varphi^*_2,\qquad
\df{J}= J + \varphi_1 \varphi^*_1 + \varphi_2 \varphi^*_2.
\end{equation}
Here we assume $\kk\notin\{-2,-3,-4\}$ since we use the presentation in \cite[Appx. A.3.4]{CFLN} which excludes these levels. Nevertheless, the following reasoning still holds when $\kk$ is equal to one of the excluded levels and can be adapted using an appropriate normalization of the strong generators.

Following \cite{KRW03}, one splits the differential $\fd{1,2}=\chid{1, 2}+\std{1, 2}$ such that
\begin{equation}
	\chid{1, 2}=\intNOz{\varphi_1^*},\qquad
	\std{1, 2}=\intNOz{v_1^+\varphi_1^*+v_2^+\varphi_2^*},
\end{equation}
are two differentials which commute with each other.
By direct computation, one checks that $\chid{1, 2}$ and $\std{1, 2}$ map the strong generators to zero except in the following cases:
\begin{align}\label{differential action I}
	\begin{split}
		&\chid{1, 2}\varphi_1 =\wun,\quad \chid{1, 2}\df{e}= \varphi^*_2, \qquad  \chid{1, 2}\df{a}= \varphi^*_1,\\
		&\std{1, 2} \varphi_1 = v^+_1,\quad \std{1, 2} \varphi_2 = v^+_2,
	\end{split}
\end{align}
and
\begin{align}
	\std{1, 2} v_1^- = &-(\kk+4) \df{L} \varphi^*_2 + (\kk+4)^2 \partial^2 \varphi^*_2 + \kk (\partial \df{e}) \varphi^*_1 - 2 \df{a} \df{b} \varphi^*_2 + \df{a}{}^2 \varphi^*_2
	\\ \notag &+ 3 \df{b}{}^2 \varphi^*_2 + ( 2 \kk + 5) (\partial \df{a})  \varphi^*_2 - ( 2 \kk + 5) (\partial \df{b}) \varphi^*_2 + (\kk+4) \df{a} \partial \varphi^*_2
	\\ \notag &- 3 (\kk+4) \df{b} \partial \varphi^*_2 + 2 \df{e} \df{f} \varphi^*_2 + 2 (\kk+4)\df{e} \partial \varphi^*_1 + 2 \df{e} \df{a} \varphi^*_1 - 2 \df{e} \df{b} \varphi^*_1,
	\\
	\std{1, 2} v_2^- = &(\kk+4) \df{L} \varphi^*_1 - (\kk+4)^2 \partial^2 \varphi^*_1 - (\kk+4) (\partial \df{f}) \varphi^*_2 + 2 \df{a} \df{b} \varphi^*_1
	\\ \notag &-3 \df{a}{}^2 \varphi^*_1 - \df{b}{}^2 \varphi^*_1 - (3\kk+7) (\partial \df{a}) \varphi^*_1 + (\kk+3) (\partial \df{b}) \varphi^*_1 - 3 (\kk+4) \df{a} \partial \varphi^*_1
	\\ \notag &+ (\kk+4) \df{b} \partial \varphi^*_1 - 2 \df{e} \df{f} \varphi^*_1 - 2 (\kk+4) \df{f} \partial \varphi^*_2 - 2 \df{f} \df{a} \varphi^*_2 + 2 \df{f} \df{b} \varphi^*_2.
\end{align}

Define the following subspaces
\begin{align}\label{first decomposition of the BRST complex}
C^\bullet_+=\langle v^+_i, \varphi_i\mid i=1,2 \rangle,\quad 
C_0^\bullet=\langle \df{a}, \df{b}, \df{e}, \df{f}, \df{L}, v_i^-, \varphi^*_i \mid i=1,2 \rangle.
\end{align}
They are closed under OPE and thus define vertex subalgebras of $C^\bullet$ whose PBW basis are obtained by restriction. 
They are also subcomplexes, i.e. 
$C^\bullet \simeq C^\bullet_+ \otimes C_0^\bullet$.
It follows from \eqref{OPEs for the natural representations} and \eqref{differential action I} that $C^\bullet_+$ is a Koszul complex and thus $\HH^n(C^\bullet_+)\simeq \delta_{n,0}\C$. 
Hence, Künneth's formula implies 
$\HH(C^\bullet)\simeq \HH(C^\bullet_0)$ as vertex algebras. 
We compute the latter by using the double complex structure on $C^\bullet_0$ as $\fd{1, 2}  = \chid{1, 2} + \std{1, 2}$.

\begin{table}[h]
\begin{center}
\scalebox{1}{
\begin{tabular}{|c|c|c|c|c|c|c|c|c|c|}
\hline
Field & $\df{a}$ & $\df{b}$ & $\df{e}$ & $\df{f}$ & $\df{L}$ &  $v_1^-$ &  $v_2^-$ & $\varphi^*_1$ &  $\varphi^*_2$
\\
\hline
$\Delta_{\mathrm{old}}$ & $1$ & $1$ & $1$ & $1$ & $2$  & $\frac{3}{2}$ &  $\frac{3}{2}$  & $1$  & $0$
\\
\hline
$\Delta_{\mathrm{new}}$ & $1$ & $1$ & $0$ & $2$ & $2$ & $2$ & $3$ & $1$ & $0$
\\
\hline
Bidegree & $(0, 0)$ & $(0,0)$ & $(0, 0)$ & $(0, 0)$ & $(0, 0)$ & $(1, -1)$ & $(1, -1)$ & $(1, 0)$ & $(1, 0)$
\\
\hline
\end{tabular}
}
\end{center}
\caption{Bidegree and conformal weights of certain strong generators of $C^\bullet_0$ with respect to the conformal vectors $L+ (\partial\varphi_1) \varphi^*_1 - \varphi_2 \partial\varphi^*_2$ ($\Delta_{\mathrm{old}}$) and $\df{L}$ ($\Delta_{\mathrm{new}}$).}
\label{tab:bigrading}
\end{table}

Let introduce a bigrading and a new conformal structure on $C^\bullet_0$ (see Table~\ref{tab:bigrading}) which give our differentials the bidegrees
\begin{align}
\bideg\, \chid{1, 2} = (1, 0),\quad \bideg\, \std{1, 2} = (0, 1).
\end{align}
Then Table \ref{tab:bigrading} above shows that the double complex is contained within the fourth quadrant, and thus, the associated spectral sequence converges.

In order to describe the cohomology $\HH^\bullet (C_0)$, we compute the second sheet of the double graded complex
\begin{align}
E_2^{\bullet,\bullet} := \HH^\bullet \left( \HH^\bullet( C_0, \chid{1, 2}), \std{1, 2}\right).
\end{align}
For this purpose, we decompose the complex $(C_0, \chid{1, 2})$ as
\begin{align}
C_0 \simeq B_0 \otimes B_1 \otimes B_2
\end{align}
where 
\begin{align}
B_0=\stgen{\df{L}, \df{b}, \df{f}, v_1^-, v_2^-},\quad B_1=\stgen{\df{a},\ \varphi^*_1},\quad B_2=\stgen{\df{e},\ \varphi^*_2}
\end{align}
define subcomplexes of $C_0$.
It follows from \eqref{differential action I} that 
\begin{align}
\HH^n(B_0,\chid{1, 2})=\delta_{n,0}B_0,\quad \HH^n(B_i,\chid{1, 2})\simeq \delta_{n,0}\C\quad (i=1,2).
\end{align}
Hence, $\HH^n(C_0,\chid{1, 2})\simeq \delta_{n,0}B_0$ by the Künneth's formula and thus the spectral sequence collapses at the second sheet, which implies the isomorphism of vector spaces
 \begin{align}
\HH^n(C_0)\simeq \delta_{n,0} \HH^0( C_0, \chid{1, 2}).
\end{align}
This implies in particular the cohomology vanishing $\fHdeg{1,2}{\neq0}(\pW{1,1,2})=0$.

Let us now focus on the 0-th cohomology which is strongly generated by the cohomology classes corresponding to $\df{L}, \df{b}, \df{f}, v_1^-, v_2^-$ that are respectively given by
\begin{align}
\df{L}, \df{b}, \df{f}, [v_1^-], [v_2^-]
\end{align}
with
\begin{align}
\begin{split}
[v_1^-] = &-v_1^- +  (\kk^2 + 4 \kk + 6) \partial^2 \df{e} - (3\kk+4) (\partial \df{e}) \df{b} - (2\kk+5) \df{e} \partial \df{b} + (2 \kk+5) (\partial \df{a}) \df{e}
\\
&+ \kk \df{a} \partial \df{e} + \df{a}{}^2 \df{e} + 3 \df{e} \df{b}{}^2 + \df{e}{}^2 \df{f} - 2 \df{a} \df{e} \df{b} - (\kk+4) \df{L} \df{e},
\\
[v_2^-] =\ &v_2^- + \frac{\kk (\kk+4)}{3\kk+8} \df{L} \df{b} -2 \df{e} \df{f} \df{b} + (\kk+4) \df{e} \partial \df{f} + \frac{1}{2} (5 \kk + 12) (\partial \df{a}) \df{a} \\
&- \frac{1}{2} \kk (\partial \df{a}) \df{b} + 2 (\kk+2) (\partial \df{e}) \df{f} - \frac{1}{2} (\kk+2) (\kk+4) \partial \df{L} + (\kk^2 + 5k + 7) \partial^2 \df{a} \\
&- \frac{3 \kk^3 + 23 \kk^2 + 57\kk+40}{3 (3\kk+8)} \partial^2 \df{b} + \frac{49 \kk^2 + 252\kk+320}{3 (3\kk+8)^2} \df{a}{}^3 \\
&- \frac{31 \kk^2 + 156\kk+192}{(3\kk+8)^2} \df{a}{}^2 \df{b} + \frac{31 \kk^2 + 156\kk+192}{(3\kk+8)^2} \df{a} \df{b}{}^2 + 2 \df{a} \df{e} \df{f} \\
&- \frac{1}{2} (\kk+4) \df{a} \partial \df{b} - \frac{29 \kk^2 + 108\kk+64}{(3\kk+8)^2} \df{b}{}^3 - (\kk+4) \df{L} \df{a}
\end{split}
\end{align}
where we temporarily exclude the level $\kk=-8/3$.
Now, it is straightforward to show that if we write $\widehat{L} = \df{L} - \frac{2}{3\kk+8} \df{b}{}^2$ then we have an isomorphism
\begin{align}
\fHdeg{1,2}{0}(\pW{1,1,2}) \rightarrow  \pW{1,3},\quad \widehat{L}, \df{f}, \df{b}, [v_1^-], [v_2^-]\mapsto L_{\text{sub}}, v^-, H, v^+, \Omega_3
\end{align}
for all levels $\kk\notin\{-2,-3,-4\}\cup\{-4,-8/3\}$, where the strong generators and their OPEs are presented in \cite{CFLN}.  
For the excluded levels including $\kk=-8/3$, the proof still holds up to a normalization of certain generators. The details of the computation are left to the reader.
This completes the proof of the first case.
\end{proof}

For the proof of the second case, there is an additional complication.
Indeed, a decomposition of the BRST complex \eqref{second BRST complex 2} similar to the previous case is not expected as the OPE between ${\Omega}_3$ with itself has a crossing term $v^+v^-$. 
We remedy this problem by first passing to the classical setting, i.e. the Poisson vertex algebras, by forming the associated graded for the Li's filtration \cite{Li05}. 
Following \cite{Ara12}, recall that given a vertex (super)algebra $V$, the Li's filtration $F^\bullet V$ is given by 
\begin{align}
    F^pV=\mathrm{span}\left\{a_{(-i_1-1)}^1\dots a_{(-i_\ell-1)}^\ell b; \begin{array}{l} a^1,\dots,a^\ell,b\in V\\ i_1,\dots i_\ell\geq0,\ i_1+\dots i_\ell\geq p \end{array} \right\}
\end{align}
for $p\geq 0$ and $F^pV=V$ for $p\leq 0$. It is decreasing $F^pV \supset F^qV$ ($p\leq q$) and, moreover, separated if $V$ is positively graded $V=\bigoplus_{\Delta>0} V_\Delta$ by conformal grading, i.e. $\cap_{p}F_pV=0$.
Then one takes the associated graded 
$$\gr^F V=\bigoplus_{p\in \Z} \gr_p^F V,\quad \gr_p^F V:=F^pV/F^{p+1}V$$
and denotes by $\sigma_p: F^p V \rightarrow \gr_p^F V$ the principal symbol maps. The space $\gr^F V$ is naturally a Poisson vertex (super)algebra with the product $\sigma_p(a)\sigma_q(b)=\sigma_{p+q}(a_{(-1)}b)$ and the $\lambda$-bracket\footnote{In the associated graded, we use the $\lambda$-bracket notation to fit the common use.}
\begin{align}
    \lbr{\sigma_p(a)}{\sigma_q(b)}=\sum_{n\geq0}\sigma_{p+q-n}(a_{(n)}b) \frac{\lambda^n}{n!}.
\end{align}
In particular, if $V$ is strongly generated by the elements $X_i$ ($i\in I$), then $\gr^F V$ is also strongly generated by the elements $\sigma_0(X_i)$ ($i\in I$), which satisfy the $\lambda$-bracket
\begin{align}
        \lbr{\sigma_0(X_i)}{\sigma_0(X_j)}= \sigma_0(X_{i(0)}X_j).
\end{align}
Indeed, the first quotient $R_V:=\gr_0^F V$, known as the Zhu's $C_2$-algebra, is a Poisson algebra generated by $\sigma_0({X}_i)$ ($i\in I$) with Poisson bracket 
\begin{equation}
    \{\sigma_0({X}_i),\sigma_0({X}_j)\}=\sigma_0(X_{i(0)}X_j)
\end{equation}
and $\gr^F V$ is a quotient of the universal Poisson vertex algebra associated with the Poisson algebra $R_V$. Note that $X_i$ ($i \in I$) strongly generate $V$ if and only if $\sigma_0({X}_i)$ ($i \in I$) generate $R_V$, see also \cite{GN03}.

\begin{proof}[Proof of \eqref{eq:case1b}]
The BRST complex in this case is given by
\begin{equation}\label{content of the BRST cpx}
\begin{gathered}
    C^{\bullet} = \pW{1, 3} \otimes \bigwedge{}^{\bullet}_{\varphi,\varphi^*},
    \quad
    \fd{2}=\intNOz{(v^+ +\wun)\varphi^*}.
\end{gathered}
\end{equation}
It is positively graded by the (symmetric) conformal grading and strongly generated at all levels. In particular, at $\kk\neq-4,-8/3$, one can chose the strong generators
\begin{align}\label{strong generators of the BRST cpx}
    L,\ H,\ \Omega_3,\ v^\pm,\ \varphi,\ \varphi^*.
\end{align}

Consider the Li's filtration of $C^{\bullet}$ and take the associated graded $\cl{C}^\bullet:=\gr^F C$. It is strongly generated by the principal symbols of the generators \eqref{strong generators of the BRST cpx}.
For convenience, denote them by the same letters, e.g. $L$ instead of $\sigma_0(L)$.
Then the differential $\overline{\fd{2}}$ on $\cl{C}^\bullet$ is given by the same formula as in \eqref{content of the BRST cpx} 
\begin{align}
    \overline{\fd{2}}=\{(v^+ +\wun)\varphi^*{}_\lambda \ \cdot\ \}|_{\lambda=0}.
\end{align} 
The spectral sequence associated with the Li's filtration $F^\bullet C^\bullet$ begins with the first sheet computing the cohomology of the associated graded
\begin{align}
E^1=\HH(\overline{C}^\bullet, \overline{\fd{2}})
\end{align}
and ends up with the associated graded of the cohomology
\begin{align}\label{spectral sequence for filtration at infty}
E^\infty=\gr^\mathcal{F}\HH({C}^\bullet, {\fd{2}})
\end{align}
where the filtration $\mathcal{F}^\bullet \HH({C}^\bullet, {\fd{2}})$ is induced from $F^\bullet C^\bullet$:
\begin{align}
\mathcal{F}^p \HH({C}^\bullet, {\fd{2}})=\mathrm{Im}\left(\ker {\fd{2}}\cap F^p {C}^\bullet \rightarrow \HH({C}^\bullet, {\fd{2}}) \right).
\end{align}
As the filtration $F^\bullet C^\bullet$ is of finite length on the homogeneous subspaces of each conformal weight, the spectral sequence converges. 

To compute the sheet $E^1$, replace the strong generators $L$, $H$, $\Omega_3$ with the following elements
\begin{align}
&{\df{L} = L + \frac{2}{3\kk+8} H^2},
\quad
\df{H} = H + \varphi \varphi^*,
\\ \notag &\df{\Omega}_3= \Omega_3 - \frac{8 (\kk+2) (\kk+4)}{3 (3\kk+8)^2} H^3 + \frac{2 (\kk+4)^2}{ (3\kk+8)} L H.
\end{align}
The non-zero $\lambda$-brackets among the strong generators are given by
\begin{align}
&\lbr{\df{H}}{\varphi} = \varphi,
\quad \lbr{\df{H}}{\varphi^*} = \varphi^*,
\quad \lbr{\varphi}{\varphi^*} = \wun, 
\quad \lbr{\df{H}}{v^{\pm}} = \pm v^{\pm},
\\ \notag
&\lbr{v^+}{v^-}= \df{\Omega}_3 + 4 \df{H}{}^3 - 12 \df{H}{}^2 \varphi \varphi^* - 2 (\kk+4)\df{H} \df{L} + 2(\kk+4) \df{L} \varphi \varphi^*.
\end{align}
Moreover, we decompose the differential $\overline{\fd{2}} = \overline{\std{2}}+ \overline{\chid{2}}$ where 
\begin{align}
\overline{\std{2}}= \lbr{v^+ \varphi^*}{\ \cdot \ }|_{\lambda=0},\quad \overline{\chid{2}}=\lbr{\varphi^*}{\ \cdot \ }|_{\lambda=0}.
\end{align}
Then $\overline{\chid{2}}, \overline{\std{2}}$ are mutually-commuting differentials which act trivially except in the following cases:
\begin{align}\label{poisson differential action II}
\begin{split}
&\overline{\chid{2}} \df{H} = \varphi^*,
\quad
\overline{\chid{2}} \varphi = \wun,
\quad
 \overline{\std{2}} \varphi = v^+,
\\
&\overline{\std{2}} v^- =  \df{\Omega}_3 \varphi^* -2(\kk+4) \df{L} \df{H} \varphi^* + 4 \df{H}{}^3 \varphi^*.
\end{split}
\end{align}

Define the Poisson vertex subalgebras 
\begin{equation}
\begin{gathered}
\overline{C}_+^\bullet = \stgen{v^+, \ \varphi},\qquad
\overline{C}_0^\bullet = \stgen{ \df{L}, \ \df{H}, \ \df{\Omega}_3, \ v^-, \ \varphi^* }.
\end{gathered}
\end{equation}
They are closed under the differentials and consequently, we have a decomposition $\overline{C}^\bullet \simeq \overline{C}_{+}^\bullet \otimes \overline{C}_0^\bullet$ into subcomplexes. Applying the same argument as the proof of \eqref{eq:case1a} to conclude $\HH^n ( \overline{C}^{\bullet} ) \simeq \HH ( \overline{C}_0^{\bullet} )$ by using $\HH^n ( \overline{C}_+^{\bullet}) \simeq \delta_{n, 0} \C$. 
Moreover, one has a double complex structure on $C^{\bullet}$, which descends to a double complex structure on $\overline{C}_0^{\bullet}$ whose bigrading is given in Table~\ref{table:SubToRegTable}. 
\begin{table}[h]
\begin{center}
\scalebox{1}{
\begin{tabular}{|c|c|c|c|c|c|}
\hline
Field & $\df{L}$ & $\df{H}$ & $\df{\Omega}_3$ & $v^-$ & $\varphi^*$
\\
\hline
$\Delta_{\text{old}}$ & $2$ & $1$ & $3$ & $2$ & $1$
\\
\hline
$\Delta_{\text{new}}$ & $2$ & $1$ & $3$ & $4$ & $1$
\\
\hline
Bidegree & $(0, 0)$ & $(0, 0)$ & $(0, 0)$ & $(1, -1)$ & $(1, 0)$
\\
\hline
\end{tabular}
}
\end{center}
\caption{Bidegree and conformal weights of certain strong generators of $C^\bullet$ with respect to the conformal vectors $L + (\partial \varphi) \varphi^*$ ($\Delta_{\mathrm{old}}$) and $\mathbf{\df{L}}$ ($\Delta_{\mathrm{new}}$) defined in \eqref{eq:strong_gen_sub_to_reg}.}
\label{table:SubToRegTable}
\end{table}

Consider the decomposition $\overline{C}_0^\bullet \simeq \overline{B}_0 \otimes \overline{B}_1$ where
\begin{align}
\overline{B}_0=\stgen{\df{L}, \df{\Omega}_3, v^-}, \quad \overline{B}_1=\stgen{\df{H},\ \varphi^*}
\end{align}
satisfy
\begin{align}\label{double complex computation}
\HH^n(\overline{B}_0,\overline{\chid{2}})=\delta_{n,0}\overline{B}_0,\quad \HH^n(\overline{B}_1,\overline{\chid{2}})\simeq \delta_{n,0}\C,
\end{align}
as follows from to \eqref{poisson differential action II}. 
Hence, the spectral sequence argument implies an isomorphism of vector spaces
\begin{align}\label{graded approximation}
\HH^n(\overline{C}_0, \overline{\fd{2}})\simeq \delta_{n,0} \HH^0( \overline{C}_0, \overline{\chid{2}}).
\end{align}
The cohomology vanishing $\HH^{\neq0}(\overline{C}^\bullet)=0$ follows and $ \fHdeg{2}{\neq0}(\pW{1,3})=0$ by \eqref{spectral sequence for filtration at infty}. 

Now, we focus on the 0-th cohomology in \eqref{graded approximation}. 
By \eqref{double complex computation}, $\df{L}, \df{\Omega}_3, v^-$ generate $\HH^0(\overline{B}_0,\overline{\chid{2}})$. Setting,
\begin{equation}
    [v^-] = v^- + (\kk+4)\df{H}{}^2 \df{L} - \df{H} \df{\Omega}_3 - \df{H}{}^4,
\end{equation}
one check that
$\df{L}, \df{\Omega}_3, [v^-]$
are strong generators of $\HH^0(\overline{C}_0)\simeq \HH^0(\overline{C})$.
By \eqref{spectral sequence for filtration at infty}, they provide an expression of the strong generators of $\gr^\mathcal{F}\HH^0({C}^\bullet, {\fd{2}})$ as cohomology classes in $\HH^0(\overline{C}^\bullet, \overline{\fd{2}})$. It is straightforward to verify that these generators of $\gr^\mathcal{F}\HH^0({C}^\bullet, {\fd{2}})$ are in fact also generators of
\begin{equation}
R_{\HH^0({C}^\bullet, {\fd{2}})} = \mathcal{F}^0 \HH^0({C}^\bullet, {\fd{2}}) / \mathcal{F}^1 \HH^0({C}^\bullet, {\fd{2}}).
\end{equation}
As mentioned previously, it was shown in \cite{GN03} that $X_i$ ($i \in I$) strongly generate a vertex algebra $V$ if and only if the corresponding principal symbols $\sigma_0 (X_i)$ ($i \in I$) generate $R_V$. Consequently $\fHdeg{2}{0}(\pW{1,3})=\HH^0({C}^\bullet, {\fd{2}})$ is strongly generated by
\begin{align}
    \mathbf{\df{L}},\ \mathbf{\Omega}_3^\natural,\ \mathbf{V}^-
\end{align}
where these fields are deformations of the principal symbol generators $\df{L}, \df{\Omega}_3, [v^-]$ given by
\begin{equation} \label{eq:strong_gen_sub_to_reg}
    \begin{aligned}
\mathbf{\df{L}} = \ &\df{L}  + 2 \partial \df{H} - (\partial \varphi) \varphi^* - 2\varphi (\partial \varphi^*),
\\ \mathbf{\Omega}_3^\natural = \ &\Omega_3^\natural -\frac{1}{3}(\kk+4)(5\kk+16)\partial^2\df{H}-2(\kk+4)\left((\partial \df{H})\df{H}- (\partial \df{H})\varphi\varphi^*-\df{H}\varphi(\partial\varphi^*)\right)
\\ &+\frac{1}{6}(\kk+2)(\kk+4)\left(\partial^2(\varphi\varphi^*)+3\partial(\varphi\partial \varphi^*) \right)+(\kk+4)(\kk+5) \varphi(\partial^2\varphi^*), 
\\ \mathbf{V}^- = \ &[v^-] -3 (\kk+3)^2 (\partial \df{H})^2 - (\kk+3)^3 \partial^3 \df{H}
\\ &- 6(\kk+3) (\partial \df{H}) \df{H}{}^2 -4(\kk+3)^2 (\partial^2 \df{H}) \df{H} + (\kk+3)(\kk+4) (\partial \df{H}) \mathbf{\df{L}}
\\ &+ \frac{1}{2} (\kk+2)(\kk+4) \df{H} (\partial \mathbf{\df{L}}).
    \end{aligned}
\end{equation}

It remains to verify that
\begin{equation}
\begin{aligned}
\pW{4} &\longrightarrow \fHdeg{2}{0}(\pW{1,3}),\\
L_{\text{reg}}&\longmapsto \mathbf{\df{L}},\\
\Omega_{3, \text{reg}}&\longmapsto \mathbf{\Omega}_3^\natural+\frac{1}{2}(\kk+4)^2\partial\mathbf{\df{L}},\\
W_4 &\longmapsto \frac{-4}{a(\kk)}\left(\mathbf{V}^--\frac{1}{4}(\kk+4)^2\mathbf{\df{L}}^2-\frac{1}{2}(\kk+3)\partial \mathbf{\Omega}_3^\natural\right.\\
&\hspace{3cm}\left.+\frac{1}{4}(\kk+3)(\kk+4)(2\kk+5)\partial^2\mathbf{\df{L}}\right)
\end{aligned}
\end{equation}
with prefactor $a(\kk)=(\kk+4) (2\kk+5) (3\kk+10)$ is an isomorphism for all levels $\kk\notin\{-4,-8/3\}\cup\{-4,-5/2,-10/3\}$, where the strong generators and their OPEs are presented in \cite{CFLN}.
This can be done by direct computation.
Note that $\mathbf{\df{L}}, \mathbf{\Omega}^\natural_3, \mathbf{V}^-$ are not closed under OPEs, due to the appearance of unavoidable crossing terms as explained at the start of the proof. However the OPEs are still closed up to $\fd{2}$-exact terms and match those given for $\pW{4}$ in \cite{CFLN} up to such terms. Finally, one can also show the isomorphism for these excluded levels as before, using an appropriate normalization.
\end{proof}

\section{Inverse Hamiltonian reduction}\label{sec:ihr}
In this section, we consider the opposite direction to partial Hamiltonian reduction which takes the form of embeddings between $\W$-algebras up to some free field algebras, that is some tensor products of the $\beta\gamma$-system and the half-lattice vertex algebra $\Pi$ (see \S \ref{sec: preliminary} for the definition).
The existence of such embeddings between $\W$-algebras associated to hook-type nilpotent orbits $\mathbb{O}_{(n,1,\ldots,1)}$ was proven by the second author \cite{Fehily2306.14673, Fehily23}.
For $\sll_4$, they are summarized as follows.
\begin{theorem}[{\cite{Fehily2306.14673, Fehily23}}]\label{iHR for hook-type}
For all levels $\kk$, there exist embeddings of vertex algebras 
\begin{align*}
&\V^\kk(\sll_4) \hookrightarrow \pW{1,1,2} \otimes (\hlat\otimes \bg^{\otimes 2}),\qquad  \pW{1,1,2}  \hookrightarrow \pW{1,3}\otimes (\hlat \otimes \bg),\\
&\pW{1,3} \hookrightarrow  \pW{4} \otimes \hlat.
\end{align*}
\end{theorem}
These three embeddings correspond to the arrows on the top line of the diagram in Fig.~\ref{fig:partial inverse HR_sl4}.
We construct the remaining arrows present on Fig.~\ref{fig:partial inverse HR_sl4}, which are stated as follows. 
\begin{theorem}\label{iHR for non hook-type}
For all levels $\kk$, there exist embeddings of vertex algebras
\begin{equation}
\pW{1,1,2} \hookrightarrow \pW{2,2} \otimes \hlat,\qquad \pW{2,2} \hookrightarrow \pW{1,3} \otimes \hlat.
\end{equation}
\end{theorem}
The first isomorphism was obtained by two of the authors by brute force computation in \cite{CFLN}. Here, we give a more conceptual proof.
As in the original cases in Theorem \ref{iHR for hook-type}, the trick to construct these embeddings is to use automorphisms of free field algebras, which transform screening operators for one $\W$-algebra to those for the other. 
In the proof of Theorem \ref{iHR for non hook-type}, the following automorphisms will be used, which can be checked by straightforward calculation.

\begin{lemma}\label{lem: nice automorphisms}\hspace{0mm}
\begin{enumerate}[wide, labelindent=0pt]
\item There exists an automorphism of vertex algebras 
$$\heis \otimes (\hlat \otimes \ff{\mathrm{1,4}}) \xrightarrow{\simeq} \heis \otimes (\hlat \otimes \ff{\mathrm{1,4}})$$
which sends the generators as follows:
\begin{equation} 
\begin{split}
   & h_1 \mapsto h_1, \quad 
    h_2 \mapsto h_2 - (\kk+4) c, \quad 
    h_3 \mapsto h_3 + (\kk+4) c, \\
   & \beta_1 \mapsto \beta_1, \quad 
    \gamma_1 \mapsto \gamma_1, \quad
    \beta_4 \mapsto -\gamma_4 \ee^{c}, \quad 
    \gamma_4 \mapsto  \beta_4 \ee^{-c}, \\
   & c \mapsto c, \quad 
    \ee^{mc} \mapsto \ee^{mc}, \quad
    d \mapsto d - \frac{1}{4}(3\kk+8) c + \frac{1}{2}h_1 + h_2 - \frac{1}{2}h_3 - 2 \beta_4 \gamma_4.
\end{split}
\end{equation}
\item There exists an automorphism of vertex algebras 
$$\heis \otimes (\hlat \otimes \ff{\mathrm{I}}) \xrightarrow{\simeq} \heis \otimes (\hlat \otimes \ff{\mathrm{I}})$$
which sends the generators as follows:
\begin{equation} 
\begin{split}
    &h_1\mapsto  h_1 - (\kk+4) c, \quad 
    h_2\mapsto  h_2 + (\kk+4) c, \quad 
    h_3\mapsto  h_3 - (\kk+4) c, \\
    &c\mapsto  c, \quad
    \ee^{m c} \mapsto  \ee^{m c}, \quad
    d \mapsto  d - (\kk+3) c + h_1 + h_3 - 2 \bt{I}\gm{I},\\
    &\bt{I} \mapsto  -\gm{I} \ee^{c},  \quad
    \gm{I} \mapsto  \bt{I} \ee^{-c}.
\end{split}
\end{equation}
\end{enumerate}
\end{lemma}

\begin{proof}[Proof of Theorem \ref{iHR for non hook-type}]
Assume first that $\kk$ is generic.
For the first embedding, we start with the realization in Proposition \ref{Several Wakimoto realizations for W-algebras} with Table \ref{tab: list of Wakimoto realizations of Walg} (I.1).
By applying the embedding \eqref{FMS} for the subalgebra $\ff{\mathrm{2}}$, we obtain the realization
\begin{align}\label{extended Wakimoto for rec?}
\pW{1,1,2}\simeq \bigcap_{i=0}^3 \ker {\s{i}{1,1,2}}\subset \heis \otimes (\hlat \otimes \ff{\mathrm{1,4}}),
\end{align}
where 
\begin{align}
\s{0}{1,1,2}=\sFMS,\quad \s{i}{1,1,2}=\int Y(P_i,z)\ \dd z
\end{align}
for $i=1,\ldots,3$ with 
\begin{align}
    P_1=\beta_1 \hwt{-\frac{1}{\kk+4}\alpha_1},\quad 
    P_2=(\underbrace{\hwt{c}}_{\beta_2} - \gamma_1 \beta_{4}) \hwt{-\frac{1}{\kk+4}\alpha_2}, \quad 
    P_3=\gamma_4 \hwt{-\frac{1}{\kk+4}\alpha_3}.
\end{align}
Now, we apply the automorphism in Lemma \ref{lem: nice automorphisms} (1), which transforms $P_i$'s as follows:
\renewcommand{\arraystretch}{1.3}
\begin{align}
\begin{array}{lll}
   P_1= \beta_{1} \hwt{-\frac{1}{\kk+4}\alpha_1} &\mapsto & \widehat{P}_1= \beta_1 \hwt{-\frac{1}{\kk+4}\alpha_1},\\
   P_2= (\wun - \gamma_1 (\beta_{4}\hwt{-c})) \hwt{-\frac{1}{\kk+4}(\alpha_2-(\kk+4)c)}&\mapsto & \widehat{P}_2=
(\wun - \gamma_1 \gamma_4) \hwt{-\frac{1}{\kk+4}\alpha_2},\\
   P_3= (\gamma_4\hwt{c}) \hwt{-\frac{1}{\kk+4}(\alpha_3+(\kk+4)c)} &\mapsto & \widehat{P}_3=-\beta_4 \hwt{-\frac{1}{\kk+4}\alpha_3}.
\end{array}
\end{align}
\renewcommand{\arraystretch}{1}
By Proposition \ref{Several Wakimoto realizations for W-algebras} with Table \ref{tab: list of Wakimoto realizations of Walg} (II.2), the screening operators $\int Y(\widehat{P}_i,z)\ \mathrm{dz}$ for $i=1,\ldots,3$ are identified with those for $\pW{1,3}$ (up to non-zero scalar) and thus the embedding \eqref{extended Wakimoto for rec?} induces
\begin{align}\label{embedding 1}
\pW{1,1,2}\hookrightarrow \pW{2,2}\otimes \hlat\subset \heis \otimes (\hlat \otimes \ff{1,4})
\end{align}
as desired. 

The proof for the second embedding is similar. Here, we start with the realization in Proposition \ref{Several Wakimoto realizations for W-algebras} with Table \ref{tab: list of Wakimoto realizations of Walg} (II.1).
By applying the automorphism \eqref{isomorphism of FFR} and then the embedding \eqref{FMS} for the subalgebra $\ff{\mathrm{II}}$, we obtain the realization
\begin{align}\label{extended Wakimoto for rec}
\pW{2,2}\simeq \bigcap_{i=0}^3 \ker {\s{i}{2,2}}\subset \heis \otimes (\hlat \otimes \ff{\mathrm{I}}),
\end{align}
where 
\begin{align}
\s{0}{2,2}=\sFMS,\quad \s{i}{2,2}=\int Y(P_i,z)\ \dd z
\end{align}
for $i=1,\ldots,3$ with 
\begin{align}
    P_1=(\bt{I}+\underbrace{\hwt{c}}_{\bt{II}})\hwt{-\frac{1}{\kk+4}\alpha_1},\quad P_2=\gm{I}\hwt{-\frac{1}{\kk+4}\alpha_2}, \quad P_3=(-\bt{I}+\underbrace{\hwt{c}}_{\bt{II}})\hwt{-\frac{1}{\kk+4}\alpha_3}.
\end{align}
The automorphism in Lemma \ref{lem: nice automorphisms} (2) transforms $P_i$'s as follows:
\renewcommand{\arraystretch}{1.3}
\begin{align}
\begin{array}{lll}
   P_1= (\wun+\bt{I}\hwt{-c})\hwt{-\frac{1}{\kk+4}(\alpha_1-(\kk+4)c)} &\mapsto & \widehat{P}_1= (\wun + \gm{I}) \ee^{-\frac{1}{\kk+4}\alpha_1},\\
   P_2= (\gm{I}\hwt{-c}) \hwt{-\frac{1}{\kk+4}(\alpha_2+(\kk+4)c)}&\mapsto & \widehat{P}_2=
 \bt{I} \hwt{-\frac{1}{\kk+4}\alpha_2},\\
   P_3= (\wun - \bt{I} \hwt{-c})\hwt{-\frac{1}{\kk+4}(\alpha_3-(\kk+4)c)} &\mapsto & \widehat{P}_3= (\wun - \gm{I})\hwt{-\frac{1}{\kk+4}\alpha_3}.
\end{array}
\end{align}
\renewcommand{\arraystretch}{1}
By Proposition \ref{inhomogeneous screening for sub}, the screening operators $\int Y(\widehat{P}_i,z)\ \mathrm{dz}$ for $i=1,\ldots,3$ are identified with those for $\pW{1,3}$ and thus the embedding \eqref{extended Wakimoto for rec} induces
\begin{align}\label{embedding 2}
\pW{2,2}\hookrightarrow \pW{1,3}\otimes \Pi \subset \heis \otimes (\hlat \otimes \ff{\mathrm{I}})
\end{align}
as desired. 

Finally, we show the embedding for all levels. Note that the Wakimoto realization of $\W$-algebras admits an integral form \cite{Gen20}, which replaces the level $\kk$ to the polynomial ring $\C[\mathbf{k}]$ with specializations $\mathbf{k}\mapsto \kk \in \C$ to recover the realization at each level. 
As for the first case, this implies that the embeddings $\{\pW{2,2}\subset \ff{1,4}\otimes \heis\}_{\kk\in\C}$ form a continuous family characterized by the kernel of screening operators \eqref{Wakimoto realizations for W-alg} at generic levels.
This is the same for the embeddings $\{\pW{1,1,2}\subset \heis \otimes (\hlat \otimes \ff{1,4})\}_{\kk\in \C}$ as Lemma \ref{lem: nice automorphisms} holds over $\C[\mathbf{k}]$. 
Since the embedding \eqref{embedding 1} for generic levels is obtained by the characterization by the same screening operators by forgetting the one corresponding to $\sFMS$ for $\pW{1,1,2}$, it extends to all levels by continuity. The second case can be treated similarly. This completes the proof. 
\end{proof}

In the proof above, the embeddings \eqref{embedding 1} and \eqref{embedding 2} factor through the kernel of the screening operator $\sFMS$.
It is clear from the construction that imposing this additional kernel characterizes the image of the embeddings at generic levels.
\begin{corollary} Let $\kk$ be generic.
\begin{enumerate}[wide, labelindent=0pt]
\item There exists an isomorphism
\begin{align}
    \pW{1,1,2} \left.\xrightarrow{\simeq}\ker \intNOz{\hwt{u_{\mathrm{min}}}}\ \right|_{\pW{2,2} \otimes \hlat} \subset \heis \otimes (\hlat \otimes \ff{\mathrm{1,4}})
\end{align} with $u_{\mathrm{min}}=\frac{1}{2}(d - \frac{1}{4}(3\kk+4)c - \frac{1}{2}h_1 - h_2 + \frac{1}{2}h_3 - 2 \beta_4 \gamma_4)$.
\item There exists an isomorphism
\begin{align} \label{eq:finalscreening}
\pW{2,2} \left.\xrightarrow{\simeq}\ker \intNOz{\hwt{u_{\mathrm{rec}}}}\ \right|_{\pW{1,3} \otimes \hlat} \subset \heis \otimes (\hlat \otimes \ff{\mathrm{I}})
\end{align}
with $u_{\mathrm{rec}}= \frac{1}{2}(d-(\kk+2)c - h_1-h_3 - 2\bt{I}\gm{I})$.
\end{enumerate} 
\end{corollary}

Moreover, combining Theorem \ref{iHR for hook-type} and \ref{iHR for non hook-type}, we obtain the following. 
\begin{corollary}\label{thm:allinvred}
Let $\lambda, \mu$ be two partitions of $4$ such that $\mathbb{O}_\lambda \geq \mathbb{O}_\mu$. Then there exists an embedding
\begin{equation} 
    \W^\kk (\OO_\mu) \hookrightarrow \W^\kk (\OO_\lambda) \otimes \beta \gamma^{N-M} \otimes \hlat^M
\end{equation}
where $M$ and $N$ are integers determined by $\lambda$ and $\mu$.
More precisely, 
\begin{equation}
    N=\frac{1}{2}(\dim \OO_\lambda-\dim \OO_\mu)
\end{equation}
and $M$ is given by the number of inverse Hamiltonian reductions used.
\end{corollary} 

\begin{remark}
    Note that in Corollary \ref{thm:allinvred}, $N$ does not depend on the path we use to relate the two $\W$-algebras $\W^\kk (\OO_\mu)$ and $\W^\kk (\OO_\lambda)$ whereas $M$ does.
    For instance, for $\lambda=(3,1)$ and $\mu=(2,1,1)$, using the direct embedding between hook-type partitions returns $M=1$, whereas going through the rectangular $\W$-algebra gives $M=2$.
    For applications to representation theory, the value of $M$ used depends on the problem at hand. For example, we choose $M$ minimal when determining whether modules induced by the inverse Hamiltonian reduction are almost irreducible \cite{AKR21}.
\end{remark}

The explicit embeddings in Theorem \ref{iHR for hook-type} and \ref{iHR for non hook-type} are helpful for studying the representation theory.
The first and last embeddings in Theorem \ref{iHR for hook-type} are explicitly described in \cite[Appx.\ A]{Fehily2306.14673} and \cite[\S 3.2]{Fehily23} respectively. The second is known to exist, and the following expressions can be obtained similarly as in the minimal-to-affine case:
\begin{equation}
\begin{gathered}
    \pW{1,1,2}  \hookrightarrow \pW{1,3}\otimes (\hlat \otimes \bg)\\
    h \mapsto  2 J - \beta \gamma -\frac{1}{2} \big( (\kk+5)c - d \big), \quad
    e \mapsto \gamma \ee^{c}, \\
    f \mapsto v^- + 2 J \beta \ee^{-c}- \frac{1}{2}(3\kk+7) \beta c \ee^{-c} + \frac{1}{2} \beta d \ee^{-c} + (\kk+2) (\partial \beta) \ee^{-c}, \\
    J \mapsto -(\kk+3) c - d - 2 \beta \gamma, \quad 
    v_1^- \mapsto \ee^{c}, \\
    L \mapsto L - \partial J - \frac{1}{2}(\partial \beta) \gamma - \frac{3}{2}\beta \partial \gamma + \frac{1}{2}cd + \frac{1}{4} (3\kk+5) \partial c-\frac{3}{4}\partial d, \\
    v_1^+ \mapsto \beta^2 \gamma^3 + (3\kk+4) \beta (\partial \gamma) \gamma + (c+d) \beta \gamma^2 + \frac{1}{2}(2\kk+5) \partial(c+d)\gamma +2(\kk+3) (\partial \beta) \gamma^2 \\ \hspace{2em}
    +(\kk+1) (\kk+2) \partial^2 \gamma +\frac{1}{4}(c+d)^2 \gamma + \frac{1}{4}(\kk+2)(3\kk+8) c^2 \gamma - v^+ \ee^{-c} \\
    \hspace{-1em}+\frac{1}{2}\big(-4(\kk+2)J + (3\kk^2+15\kk+16)c + \kk d \big)\partial \gamma \\ \hspace{2em}
    -(\kk+4) L \gamma + v^- \gamma^2 \ee^c + 2 \big(J-(\kk+2)c \big) J \gamma
\end{gathered}
\end{equation}
and the expressions of the generators $v^\pm_2$ are deduced using the OPEs.
The formula for the first embedding in Theorem \ref{iHR for non hook-type} is found in \cite[Proposition 4.5, Remark 4.9]{CFLN}. The following formula for the second embedding is derived by direct computation:
\begin{gather}
\pW{2,2} \hookrightarrow \pW{1,3} \otimes \hlat\\
\begin{split}
    &e  \mapsto \ee^c,\quad
    h  \mapsto 2b,\quad
    L_{\text{tot}}    \mapsto L + \frac{1}{2}c d -\partial a,\quad
    v_2^+   \mapsto -\Big(J + \frac{3\kk+8}{2(\kk+2)} a \Big) \ee^c,\\
    &f  \mapsto \Big( (\kk+4)L - 2 J^2 + \frac{1}{4}v^+ 
     - \frac{1}{4}v^- - a^2 - (2\kk+5) \partial a \Big)\ee^{-c}.
\end{split}
\end{gather}
where $L_\text{tot} = L_1 + L_2 + \frac{1}{2(\kk+2)}J^2$ and we set 
\begin{equation}
\begin{gathered}
    a = \frac{1}{2}\big( -(\kk+2)c + d \big), \quad
    b = \frac{1}{2}\big( (\kk+2)c + d \big).
\end{gathered}
\end{equation}
The images of the remaining generators of $\pW{2,2}$ can be obtained from these by imposing the relevant OPEs.

\section{Partial reductions for modules in Kazhdan--Lusztig category}\label{sec: Partial reductions for modules}
Let us assume $\kk$ to be generic. 
Let $\KL^\kk(\g)$ denote the  Kazhdan--Lusztig category, i.e., the category of $\V^\kk(\g)$-modules which are bounded from below by the conformal grading and with finite-dimensional graded spaces.
As $\kk$ is generic, $\KL^\kk(\g)$ is semisimple and the simple modules are the Weyl modules 
\begin{equation}
    \weyl_\lambda^\kk=U(\widehat{\g})\otimes_{U(\g[t]\oplus\C K)}L_\lambda,
\end{equation}
which are induced from the simple $\g$-modules $L_\lambda$ of dominant integral highest weights $\lambda \in P_+$.

The Weyl modules $\weyl_\lambda^\kk$ admit resolutions by Wakimoto modules of the form
\begin{align}\label{Wakimoto resolution of Weyl modules}
    0\rightarrow  \weyl_{\lambda}^\kk \rightarrow \affWak{\lambda}\overset{\bigoplus S_{i,\lambda}}{\longrightarrow} \bigoplus_{i=1,\ldots,3} \affWak{s_i\circ \lambda}\rightarrow \mathrm{C}_2^\lambda \rightarrow \cdots \rightarrow \mathrm{C}_N^\lambda \rightarrow 0
\end{align}
with 
\begin{align}
    \mathrm{C}_i^\lambda=\bigoplus_{\begin{subarray}c w\in W\\ \ell(w)=i\end{subarray}} \affWak{w\circ \lambda},
\end{align}
which generalize the case $\lambda=0$ in \eqref{Wakimoto resolution of affine}.
Here $S_{i,\lambda}$ are defined as $S_{i,\lambda}=S_{i}^{[h_i(\lambda)+1]}$ where we set
\begin{align}\label{product of screening operators}
    S_{i}^{[n]}=\int_{\Gamma} S_i(z_1)\dots S_{i}(z_{n})\ \dd z_1\dots \dd z_{n}\colon \affWak{\lambda}\rightarrow \affWak{\lambda-n \alpha_i}
\end{align}
in general for some local system $\Gamma$ on the configuration space $Y_{n}=\{(z_1,\dots,z_{n})\mid z_p\neq z_q\}$.
The following proposition describes the induced homomorphism $[S_{i}^{[n]}]$ when applying the reduction $\HH^0_{\OO_f}$ to \eqref{product of screening operators}.
\begin{proposition}\label{prop: screening for modules}
The following diagram commutes
\begin{center}
\begin{tikzcd}
\HH^0_{\OO_f}(\affWak{\lambda}) \arrow{rr}{[S_{i}^{[n]}]} \arrow{d}{\simeq} && \HH^0_{\OO_f}(\affWak{\lambda-n \alpha_i}) \arrow{d}{\simeq}\\
\affWak{\lambda,f} \arrow{rr}{S_{i}^{f,[n]}} && \affWak{\lambda-n \alpha_i,f}
\end{tikzcd}
\end{center}
where the vertical isomorphisms are given by \eqref{Wakimoto for Walg} and $S_{i}^{f,[n]}$ is given by the formula
\begin{align}
    S_{i}^{f,[n]}=\int_{\Gamma} S^f_i(z_1)\cdots S^f_{i}(z_{n})\ \dd z_1\cdots \dd z_{n}\colon \affWak{\lambda,f}\rightarrow \affWak{\lambda-n \alpha_i,f}.
\end{align}
\end{proposition}
\begin{proof}
Recall that the quantum Hamiltonian reduction $\HH_{\OO_f}(M)$ can be written in terms of the semi-infinite cohomology for the loop algebra $\g_+(\!(z)\!)$:
\begin{align}
    \HH_{\OO_f}(M)=\HH\left(M\otimes \bigwedge{}^{\bullet}_{\g_+}, d \right)
    =\HH^{\frac{\infty}{2}}(\g_+(\!(z)\!);M\otimes \C_\chi)
\end{align}
where $\C_\chi$ is the one-dimensional $\g_+(\!(z)\!)$-module on which the currents $a(z)$ ($a\in \g_+$) act by pairing $a(z)\mapsto (f,a)$ and $M\otimes \C_\chi$ is regarded as the tensor representation over $\g_+(\!(z)\!)$.
By taking $M=\affWak{\lambda}$ and the good coordinates using the decomposition $N_+\simeq G_{0,+} \ltimes G_+$ as in \S \ref{sec:ffr}, we have 
\begin{align}
    \HH_{\OO_f}(\affWak{\lambda})\simeq \HH^{\frac{\infty}{2}}(\g_+(\!(z)\!);\bg^\blacktriangle\otimes \C_\chi)\otimes (\bg^\bigstar\otimes \Fock{\lambda})
\end{align}
with $\blacktriangle= \{i=1,\ldots,6 ; \Gamma(\alpha_i)>0\}$.
The free field algebra $\bg^\blacktriangle$ has the structure of a $\g_+(\!(z)\!)$-module called a semi-regular bimodule \cite{Ara14, Voro} whose module structure is given by the homomorphism $\Psi\colon \V^0(\g_+) \rightarrow \bg^\blacktriangle$ together with the anti-homomorphism 
$\widehat{\rho}\colon \V^0(\g_+) \rightarrow \bg^\blacktriangle$
induced from the left and right multiplications of $G_+$ on itself as in \S \ref{sec:ffr}. By \cite{Ara14, ACL19}, there exists a vector space isomorphism 
\begin{align}
   \xi\colon \bg^\blacktriangle\otimes \C_\chi\xrightarrow{\simeq} \bg^\blacktriangle\otimes \C_\chi
\end{align}
which intertwines the actions
\begin{align}\label{switching property}
\begin{split}
 &(\xi \circ \sigma)(a(z))= (a(z)\otimes 1) \circ \xi,\\
 &\xi \circ \left(\widehat{\rho}(a(z))\otimes 1\right)=\left(\widehat{\rho}(a(z))\otimes 1+1\otimes a(z)\right)\circ \xi
\end{split}
\end{align}
for $a\in \g_+$ where $\sigma$ is the coproduct $\sigma(a(z))=a(z)\otimes 1+1\otimes a(z)$.
Then $\xi$ induces an isomorphism  
\begin{align}
   \HH^{\frac{\infty}{2}}(\g_+(\!(z)\!);\bg^\blacktriangle\otimes \C_\chi) \xrightarrow{\simeq} \HH^{\frac{\infty}{2}}(\g_+(\!(z)\!);\bg^\blacktriangle)\otimes \C_\chi \simeq \C_\chi=\C.
\end{align}
The advantage here is that the higher conformal weight subspaces of $\bg^\blacktriangle$ for $\HH^{\frac{\infty}{2}}(\g_+(\!(z)\!);\bg^\blacktriangle)$ are contracted to be zero. 
Hence, under this isomorphism,
\begin{equation}
    \begin{aligned}
&\longstick{[S_{i}^{[n]}]}{\HH^0_{\OO_f}(\affWak{\lambda})}\\
&=\longstick{\int_{\Gamma} \prod_{j=1}^n Y(\widehat{\rho}(e_i)\fockIO{i},z_j)\ \dd z_1\cdots \dd z_{n}}{\HH^{\frac{\infty}{2}}(\g_+(\!(z)\!);\bg^\blacktriangle\otimes \C_\chi) \otimes (\bg^\bigstar\otimes \Fock{\lambda})}\\
&\overset{\eqref{switching property}}{=}
\longstick{\int_{\Gamma} \prod_{j=1}^n Y(\widehat{\rho}_{f,\Gamma}(e_i)\fockIO{i},z_j)\ \dd z_1\cdots \dd z_{n}}{\C \otimes (\bg^\bigstar\otimes \Fock{\lambda})}\\
&=\longstick{\int_{\Gamma} S^f_i(z_1)\cdots S^f_{i}(z_{n})\dd z_1\cdots \dd z_{n}}{\affWak{\lambda,f}}\\
&=\longstick{S_{i}^{f,[n]}}{\affWak{\lambda,f}}
    \end{aligned}
\end{equation}
as desired.
\end{proof}
\begin{remark}
    The proof of Proposition \ref{prop: screening for modules} generalizes to higher ranks for $\W$-algebras admitting even good gradings.
\end{remark}

\begin{corollary}\label{Free field realization of W-modules}
    For $\lambda\in P_+$, there is an isomorphism of $\W^\kk(\OO_f)$-modules
    \begin{align*}
        \HH^0_{\OO_f}(\weyl_\lambda^\kk)\simeq \bigcap_{i=1}^3 \ker \longstick{S_{i,\lambda}^f}{\affWak{\lambda,f}},\quad S_{i,\lambda}^f:=S_{i,f}^{[h_i(\lambda)+1]}.
    \end{align*}
\end{corollary}
Thanks to this free field realization, we may generalize the results on $\W$-algebras in \S \ref{Sec: Partial reductions} to their modules.

\begin{theorem}\label{Isom for the modules 1} \hspace{0mm}
The $\W^\kk({\OO_f})$-modules $\HH_{\OO_f}^0(\weyl^\kk_\lambda)$ $(\lambda\in P_+)$ are simple at generic levels $\kk$ and satisfy the following isomorphisms:
\begin{align}
\begin{array}{lll}
\fHdeg{2,2}{0}(\weyl_\lambda^k)\simeq \fHdeg{2}{0}(\fHdeg{1,1,2}{0}(\weyl_\lambda^k)),&& \fHdeg{1,3}{0}(\weyl_\lambda^k)\simeq \fHdeg{2}{0}(\fHdeg{2,2}{0}(\weyl_\lambda^k)),\\
&&\\
\fHdeg{4}{0}(\weyl_\lambda^k)\simeq \fHdeg{2}{0}(\fHdeg{1,3}{0}(\weyl_\lambda^k)),&& \fHdeg{1,3}{0}(\weyl_\lambda^k)\simeq \fHdeg{1,2}{0}(\fHdeg{1,1,2}{0}(\weyl_\lambda^k)).
\end{array}
\end{align}
\end{theorem}
\begin{proof}
Let us show the first isomorphism, i.e. 
\begin{align}
    \fHdeg{2,2}{0}(\weyl_\lambda^k)\simeq \fHdeg{2}{0}(\fHdeg{1,1,2}{0}(\weyl_\lambda^k)).
\end{align}
Thanks to Corollary \ref{Free field realization of W-modules}, the isomorphism follows from the identification of the screening operators appearing in the proofs of Theorem \ref{thm:min_to_rect} and Proposition \ref{prop: screening for modules} as follows:
\begin{equation}
    \begin{split}
        &\fHdeg{2}{0}(\fHdeg{1,1,2}{0}(\weyl_\lambda^k))
        \simeq \bigcap_{i=1}^3 \ker \longstick{[S_{i,\lambda}^{\ \ydiagram{1,1,2}}]}{\fHdeg{2}{0}(\pWak{1,1,2}{\lambda})}\\
        &\hspace{1cm}\simeq \bigcap_{i=1}^3 \ker \longstick{\widetilde{S}_{i,\lambda}^{\ \ydiagram{1,1,2}}}{\ff{1,4}\otimes \heis}
        \overset{\iota}{\simeq}\bigcap_{i=1}^3 \ker \longstick{ {S}_{i,\lambda}^{\ \ydiagram{2,2}}}{\ff{1,3}\otimes \heis}
        \simeq \fHdeg{2,2}{0}(\weyl_\lambda^k)
    \end{split}
\end{equation}
where
\begin{align}
    \widetilde{S}_{i,\lambda}^{\ \ydiagram{1,1,2}}=\int_{\Gamma} \prod_{j=1}^n \ts{i}{1,1,2}(z_j)\ \dd z_1\cdots \dd z_{n}
\end{align}
with 
\begin{align}
  n=h_i(\lambda)+1,\quad \ts{i}{1,1,2}(z):=Y(P_i \fockIO,z),
\end{align}
see \eqref{coefficients in the second reduction} for $P_i$'s. This completes the proof for the first isomorphism.
The remaining cases are shown in the same manner by using the identification of the screening operators obtained in the proof of Theorem \ref{thm: pqhrRectoSub} and \ref{thm:min_to_rect}.
\end{proof}

\printbibliography
\end{document}